\documentclass[10pt,a4paper]{amsart}
\usepackage{amsmath, graphicx, subfigure, epsfig,amssymb}
\usepackage{xcolor}
\usepackage[normalem]{ulem}
\numberwithin{equation}{section}

\newcommand{\e}{\epsilon}
\newcommand{\ga}{\gamma}
\newcommand{\de}{\delta}
\newcommand{\br}{\mathbb{R}}

\newcommand{\ik}{\varphi}
\newcommand{\pa}{\partial}
\newcommand{\bt}{\beta}
\newcommand{\al}{\alpha}
\newcommand{\la}{\lambda}

\newcommand{\be}{\begin{equation}}
\newcommand{\ee}{\end{equation}}

\newcommand{\tth}{\tilde\nu}
\newcommand{\auxang}{\nu}
\newcommand{\dir}{\Theta}
\newcommand{\dd}{\text{d}}
\newcommand{\frec}{f_\e^{\text{rec}}}

\newcommand{\bma}{\begin{pmatrix}}
\newcommand{\ema}{\end{pmatrix}}

\newcommand{\s}{\mathcal S}

\newcommand{\gl}{g_l}
\newcommand{\gc}{g}

\newcommand{\CH}{\mathcal H}

\newcommand{\Psil}{\Psi_l}

\newcommand{\DTB}{\text{DTB}}

\newtheorem{theorem}{Theorem}

\newtheorem{lemma}{Lemma}

\newtheorem{definition}{Definition}
\newtheorem*{conjecture}{Conjecture}

\begin{document}

\title[Analysis of resolution]{Resolution of 2D reconstruction of functions with nonsmooth edges from discrete Radon transform data}
\author[A Katsevich]{Alexander Katsevich$^1$}
\thanks{$^1$This work was supported in part by NSF grant DMS-1906361. Department of Mathematics, University of Central Florida, Orlando, FL 32816 (Alexander.Katsevich@ucf.edu). }

\begin{abstract} 
Let $f$ be an unknown function in $\mathbb R^2$, and $f_\epsilon$ be its reconstruction from discrete Radon transform data, where $\epsilon$ is the data sampling rate. We study the resolution of reconstruction when $f$ has a jump discontinuity along a nonsmooth curve $\mathcal S_\epsilon$. The assumptions are that (a) $\mathcal S_\epsilon$ is an $O(\epsilon)$-size perturbation of a smooth curve $\mathcal S$, and (b) $\mathcal S_\epsilon$ is Holder continuous with some exponent $\gamma\in(0,1]$. We compute the Discrete Transition Behavior (or, DTB) defined as the limit $\text{DTB}(\check x):=\lim_{\epsilon\to0}f_\epsilon(x_0+\epsilon\check x)$, where $x_0$ is generic. We illustrate the DTB by two sets of numerical experiments. In the first set, the perturbation is a smooth, rapidly oscillating sinusoid, and in the second - a fractal curve. The experiments reveal that the match between the DTB and reconstruction is worse as $\mathcal S_\epsilon$ gets more rough. This is in agreement with the proof of the DTB, which suggests that the rate of convergence to the limit is $O(\epsilon^{\gamma/2})$. We then propose a new DTB, which exhibits an excellent agreement with reconstructions. Investigation of this phenomenon requires computing the rate of convergence for the new DTB. This, in turn, requires completely new approaches. We obtain a partial result along these lines and formulate a conjecture that the rate of convergence of the new DTB is $O(\epsilon^{1/2}\ln(1/\epsilon))$.
\end{abstract}
\maketitle

\section{Introduction}\label{sec_intro}

Action of the Radon transform (and, by extension, its inverse) on distributions is a topic that received considerable attention over the years \cite{gs5, hel2, rk}. From a practical perspective, if an object to be reconstructed has singularities (e.g., jumps or edges), it is important to know how well (e.g., with what resolution) these singularities can be reconstructed when the data are discrete. Convergence of numerical Radon inversion algorithms for non-smooth functions has been studied as well \cite{gonchar86,  pal93, popov90, popov98}. In these works the discontinuities of the object are a complicating factor rather than the object of study. 


Let $f$ represent the unknown function, and $\s$ denote its singular support. Let $\check f$ be a reconstruction from continuous data, and $\check f_\e$ -- the corresponding reconstruction from discrete data, where $\e$ represents the data sampling rate. In the latter case, interpolated discrete data are substituted into the ``continuous'' inversion formula. When theoretically exact reconstruction is desired, $\check f\equiv f$. Generally, $\check f$ does not coincide with $f$. For example, one can be interested in edge-enhanced reconstruction, as in local tomography \cite{fbh01, rk} or when computing  derivatives of $f$ directly from the data \cite{hl12, louis16}. In other cases, e.g., for more general Radon transforms, an exact inversion formula may not exist. In this case one usually reconstructs $f$ modulo less singular terms, i.e. $f-\check f$ is smoother than $f$. 

In \cite{Katsevich2017a, kat19a, kat20b, kat20a, Kats_21resol} the author developed the analysis of reconstruction,  called {\it local resolution analysis}, by focusing specifically on the behavior of $\check f_\e$ near $\s$. One of the main results of these papers is the computation of the limit 
\be\label{tr-beh}
\DTB(\check x)
:=\lim_{\e\to0}\e^\kappa \check f_\e(x_0+\e\check x)
\ee
in a variety of settings.
Here $x_0\in\s$ is generic (see Definition~\ref{def:gen pt} below), $\kappa\ge0$ is selected based on the strength of the singularity of $\check f$ at $x_0$, and $\check x$ is confined to a bounded set. It is important to emphasize that both the size of the neighborhood around $x_0$ and the data sampling rate go to zero simultaneously in \eqref{tr-beh}. The limiting function $\DTB(\check x)$, which we call the discrete transition behavior (or DTB for short), contains complete information about the resolution of reconstruction.  

The practical use of the DTB is based on the relation
\be\label{DTB orig use}
\check f_\e(x_0+\e\check x)=\e^{-\kappa}\DTB(\check x)+\text{error term}.
\ee
When $\e>0$ is sufficiently small, the error term is negligible, and $\e^{-\kappa}\DTB(\check x)$, which is typically computed by a simple formula, is an accurate approximation to the numerical reconstruction.

The functions, which have been investigated in the framework of local resolution analysis so far, are conormal distributions, whose wave front set coincides with the conormal bundle of a smooth surface. To put it another way, these distributions are nonsmooth across a smooth surface, and are smooth along it. On the other hand, in many applications the discontinuities of a sample $f$ occur across non-smooth (rough) surfaces. Examples include soil and rock imaging, where the surface of cracks and pores is highly irregular and frequently simulated by fractals \cite{Anovitz2015, GouyetRosso1996, Li2019, soilfractals2000, PowerTullis1991, Renard2004, Zhu2019}. 

Micro-CT (i.e., CT capable of achieving micrometer resolution) is an important tool for imaging of rock samples extracted from the well. Larger samples are called rock cores, and smaller samples are called plugs. As stated in \cite{Zhu2019}, ``The simulation of various rock properties based on three-dimensional digital cores plays an increasingly important role in oil and gas exploration and development. The accuracy of 3D digital core reconstruction is important for determining rock properties.'' Here the term ``digital core'' refers to a digital representation of the rock core obtained, for example, as a result of CT or micro-CT scanning and reconstruction.
%
Accurate identification of the pore space inside rock samples is of utmost importance because it contributes to accurate estimation of the amount of hydrocarbon reserves in a given formation and brings many additional benefits. As stated above, the boundary between the solid matrix and the pore space is typically rough (see also \cite{Cherk2000}), i.e., it contains features across a wide range of scales, including the scales below what is accessible with micro-CT. Therefore the effects that degrade the resolution of micro-CT (e.g., the partial volume effect due to finite data sampling) and how these effects manifest themselves in the presence of rough boundaries require careful investigation. Once fully understood and quantified, these effects can be accounted for to improve pore space determination when analysing the reconstructed images.

Very little is known about how the Radon transform acts on distributions with more complicated singularities. A recent literature search reveals a small number of works, which investigate the Radon transform acting on random fields \cite{JainAnsary1984, Sanz1988, Medina2020}. For example, the author did not find any publication on the Radon transform of characteristic functions of domains with rough boundaries. This appears to be the first paper on the Radon transform of functions with rough edges.

In this paper we use the local resolution analysis in $\br^2$ to study the resolution of reconstruction when $f$ has a jump discontinuity across a nonsmooth curve $\s_\e$. Exact reconstruction from the classical Radon transform data is considered, i.e. $\check f(x)\equiv f(x)$ and $\kappa=0$. 
The assumption is that $\s_\e$ is Holder continuous with some exponent $\ga\in(0,1]$. We assume that $\s_\e$ is a small perturbation of a smooth curve $\s$. The perturbation is of size $O(\e)$ along the direction normal to $\s$, and the perturbation scales like $O(\e^{1/2})$ along the direction tangential to $\s$. Due to the linearity of the Radon transform, we can assume that $f$ is supported in the narrow domain bounded by $\s$ and $\s_\e$. The function supported in this domain is denoted $f_\e(x)$, and the reconstruction of $f_\e(x)$ from discrete data is denoted $\frec(x)$. We obtain the DTB and illustrate it by two sets of numerical experiments. In the first set, the perturbation $\s\to\s_\e$ is a smooth sinusoid with amplitude $O(\e)$ and period $O(\e^{1/2})$. 
Results of these experiments with $\e=\e_1=1.2/500$ and $\e=\e_2=1.2/1000$ demonstrate a good agreement between the DTB and reconstruction. 

The second set involves a fractal perturbation, which is specified in terms of the Weiertsrass-Mandlebrot function \cite{BerryLewis1980, Borodich1999}. As before, the magnitude of the perturbation is $O(\e)$, and it scales like $O(\e^{1/2})$ along $\s$. Its Holder exponent is $\ga=1/2$. It turns out that the match between the DTB and reconstruction is now much worse than before for the same two values $\e=\e_{1,2}$. Note that the DTB is an accurate approximation to the reconstruction only when $\e>0$ is sufficiently small. Analysis of the derivation of the DTB suggests (but not proves) that the rate of convergence in \eqref{tr-beh} is $O(\e^{\ga/2})$ (which is also the magnitude of the error term in \eqref{DTB orig use}). In other words, the rougher $\s_\e$ is, the slower the convergence and the larger the error. Therefore, to obtain a good match when $\s_\e$ is fractal a much smaller value $\e\ll\e_2$ should be used.

Analysis of the reconstruction formula reveals a potentially more accurate expression for $\frec(x_0+\e\check x)$. Even though the DTB was originally defined as the limit in \eqref{tr-beh}, with a slight abuse of notation, any easily computable approximation to $\frec(x_0+\e\check x)$ will be called a DTB as well and denoted $\DTB_{new}$. In particular, $\DTB_{new}$ may have a more complicated $\e$-dependence than  the one in \eqref{tr-beh}:
\be\label{DTB new use}
\frec(x_0+\e\check x)=\DTB_{new}(\check x,\e)+\text{error term}.
\ee
The idea is that by allowing a more general $\e$-dependence, the error term in \eqref{DTB new use} can be smaller than the one in \eqref{DTB orig use}.

Numerical experiments with the new DTB show a perfect match between $\DTB_{new}$ and reconstruction for the two values $\e=\e_{1,2}$ used before. The two results do not contradict each other, because the original, less accurate DTB is a small-$\e$ limit of the new, more accurate DTB. 

Rigorous derivation of $\DTB_{new}$ is significantly more difficult than that of the original one. Even the proof of the original DTB \eqref{tr-beh} (see Sections~\ref{sec:beg proof}--\ref{sec: rem sing}) establishes the existence of convergence, but not its rate (see the last paragraph in Subsection~\ref{ssec:f11}). To prove that $\DTB_{new}$ is indeed more accurate, one needs to estimate its approximation error (and that of the original DTB). We distinguish two cases: $x_0\in\s$ and $x_0\not\in\s$, and prove that in the second case, assuming in addition that there is no line through $x_0$ which is tangent to $\s$, the rate of convergence of $\DTB_{new}$ is $O(\e^{1/2}\ln(1/\e))$. Based on this result and numerical evidence, we formulate the conjecture that the same rate holds in the remaining, unproven cases. Our proof of the second case uses different tools and, at its core, uses a phenomenon different from the one in the proof of the original DTB. The proof of the remaining cases is difficult,  requires entirely new approaches, and is outside the scope of this paper. 

The new DTB (see eq. \eqref{new DTB} below) is given by a convolution of an explicitly computed and suitably scaled kernel with $f_\e$. Thus, it can be used quite easily to investigate partial volume effects and resolution in the case of rough (e.g., fractal) boundaries. Superficially, this kernel resembles the point spread function (PSF) of filtered backprojection reconstruction \cite[Section 12.3]{eps08}. Nevertheless, its origin, use (analysis of reconstruction in a neighborhood of a singularity of $f$), and method of proof are all completely different from those for the PSF.

To summarize, the main results of the paper are as follows:
\begin{enumerate}
\item Derivation of the original DTB;
\item Numerical demonstration that the accuracy of the original DTB drops as the curve $\s_\e$, across which $f$ is discontinuous, becomes less smooth (fractal);
\item A new DTB is proposed, which is shown numerically to be much more accurate than the original one for fractal $\s_\e$; and
\item A conjecture about the accuracy of the new DTB and its proof in the case $x_0\not\in\s$ under some additional assumptions.
\end{enumerate}

The paper is organized as follows. In Section~\ref{sec:a-prelims}, we describe the problem setup, state the relevant result from our earlier paper \cite{Katsevich2017a} as Theorem~\ref{lem-phi-lim}, and formulate the first result of this paper - a formula for the original DTB \eqref{tr-beh} - as Theorem~\ref{main-res}.  Section~\ref{sec:beg proof} contains most of the proof of the theorem. 
In Section~\ref{sec: dtb ex} we compute the original DTB explicitly and consider two theoretical examples. In the first one the perturbation is a constant function along $\s$, and we recover the result of \cite{Katsevich2017a}. In the second example we consider a fractal perturbation specified in terms of a Weiertsrass-Mandlebrot function. The fact that the perturbed boundary $\s_\e$ does not create non-local artifacts is proven in Section~\ref{sec: rem sing}, thereby finishing the proof of Theorem~\ref{main-res}. Numerical experiments with the original DTB are in Section~\ref{sec numexp I}, where oscillatory and fractal perturbations of $\s$ are considered. In Section~\ref{sec numexp II} we describe $\DTB_{new}$ and present numerical experiments with the new formula. The experiments demonstrate improved accuracy for fractal $\s_\e$. We formulate a conjecture about the accuracy of $\DTB_{new}$, and state Lemma~\ref{lem:partial res} about the magnitude of error when $x_0\not\in\s$. The proof of the lemma is in Appendix~\ref{sec new ker prf}. Let $H_0(s)$ be a function that describes the perturbation $\s\to\s_\e$ (after appropriate rescaling). One of the main assumptions on the perturbation is that the level sets of $H_0$, i.e. the sets $H_0^{-1}(\hat t)$, are well-behaved for any $\hat t\in H_0(\br)$. In Appendix~\ref{sec:bad fn} we construct a function on $\br$, whose level sets are well-behaved as required for the lemma, which is Holder continuous with exponent $\ga$ for any prescribed $0<\ga<1$, but which is not Holder continuous with any exponent $\ga'>\ga$ on a dense subset of $\br$. The proofs of two auxiliary lemmas are in Appendices~\ref{sec:lemI} and \ref{sec:prf-lemPsi}.

\section{Preliminaries}\label{sec:a-prelims}

Consider a compactly supported function $f(x)$ on the plane, $x\in\br^2$. Set $\s:=\{x\in\br^2:f\not\in C^2(U)\text{ for any open }U\ni x\}$. We suppose that 
\begin{itemize}
\item[f1.] For each $x_0\in\s$ there exist a neighborhood $U\ni x_0$, domains $D_\pm$, and functions $f_\pm\in C^2(\br^2)$ such that
\begin{equation}\label{f_def}\begin{split}
& f(x)=\chi_{D_-}(x) f_-(x)+\chi_{D_+}(x) f_+(x),\ x\in U\setminus \s,\\
& D_-\cap D_+=\varnothing,\ D_-\cup D_+=U\setminus \s,
\end{split}
\end{equation}
where $\chi_{D_\pm}$ are the characteristic functions of $D_\pm$,  
\item[f2.] $\s$ is a $C^4$ curve;
\item[f3.]\label{ass:curv} There are finitely many points $x\in \s$ where the curvature of $\s$ equals zero, and these zeroes are of finite order.
\end{itemize}

The discrete tomographic data are given by
\be\label{data_eps}
\hat f_\e(\al_k,p_j):=\frac1\e \iint w\left(\frac{p_j-\vec\al_k\cdot y}{\e}\right)f(y)\dd y,\
p_j=j\Delta p,\ \al_k=k\Delta\al,
\ee
where $w$ is the detector aperture function, $\Delta p=\e$, $\Delta\al=\kappa\e$, and $\kappa>0$ is fixed. 
Here and below, $\vec \al$ and $\al$ in the same equation are always related by $\vec\al=(\cos\al,\sin\al)$. The same applies to $\vec\Theta=(\cos\theta,\sin\theta)$ and $\theta$.

\noindent{\bf Assumptions about the aperture function $w$:}
\begin{itemize}
\item[AF1.] $w$ is even and $w\in C_0^2(\br)$ (i.e., $w$ is compactly supported, and $w''\in L^\infty(\br)$); and 
\item[AF2.] $\int w(p)dp=1$.
\end{itemize}
Reconstruction from discrete data is achieved by the formula
\be\label{recon-orig}
\frec(x)=-\frac{\Delta\al}{2\pi}\sum_{|\al_k|\le \pi/2} \frac1\pi \int \frac{\pa_p\sum_j\ik\left(\frac{p-p_j}\e\right)\hat f_\e(\al_k,p_j)}{p-\al_k\cdot x}\dd p,
\ee
where $\ik$ is an interpolation kernel.

\noindent{\bf Assumptions about the interpolation kernel $\ik$:}
\begin{itemize}
\item[IK1.] $\ik$ is even and $\ik\in C_0^2(\br)$;
\item[IK2.] $\ik$ is exact up to order $1$, i.e. 
\be\label{exactness}
\sum_{j\in\mathbb Z} j^m\ik(u-j)\equiv u^m,\ m=0,1,\ u\in\br.
\ee
\end{itemize}
As is easily seen, assumption IK2 implies $\int\ik(p)dp=1$.

\begin{definition}\cite{Katsevich2017a}\label{def:gen pt}
A point $x_0\in\br^2$ is generic if
\begin{enumerate}
\item No line, which is tangent to $\s$ at a point where the curvature of $\s$ is zero, passes through $x_0$;
\item If $x_0\in\s$, the quantity $\kappa x_0\cdot \vec\tau$ is irrational, where $\vec\tau$ is a unit tangent vector to $\s$ at $x_0$. 
\end{enumerate}
\end{definition}
\noindent
Condition (1) in the definition implies that the curvature of $\s$ at $x_0$ is nonzero if $x_0\in\s$.

Pick a generic point $x_0\in\s$. Let $\vec\dir_0$ be the unit normal to $\s$ at $x_0$, which points from $x_0$ towards the center of curvature of $\s$ at $x_0$. We will call the side of $\s$ where $\vec\dir_0$ points ``positive'', and the opposite side - ``negative''. Without loss of generality, we can assume in \eqref{f_def} that $D_+$ is on the positive side of $\s$, and $D_-$ is on the negative side. We formulate here the relevant result from \cite{Katsevich2017a}.

\begin{theorem}[\cite{Katsevich2017a}]\label{lem-phi-lim} Let (a) $f$ satisfy conditions f1--f3; (b) interpolation kernel $\ik$ satisfy conditions IK1, IK2; and (c) detector aperture function $w$ satisfy conditions AF1, AF2. Suppose $x_0\in\s$ is generic, and let $\vec\dir_0$ be the positive unit normal to $\s$ at $x_0$. If $\frec$ is the reconstruction of the original, unperturbed $f$ from the data \eqref{data_eps} using \eqref{recon-orig}, then 
\be\label{final-lim-v1}
\lim_{\e\to0}\frec(x_0+\e\hat x)= f_-(x_0)+(f_+(x_0)-f_-(x_0))\int_{-\infty}^{\vec\Theta_0\cdot\hat x}(\ik*w)(r)dr,
\ee
where $f_\pm(x_0)$ are the same as in \eqref{f_def}. If $x_0\not\in\s$ is generic, then 
\be\label{final-lim far zone}
\lim_{\e\to0}\frec(x_0+\e\hat x)= 0.
\ee
\end{theorem}
The result in \cite{Katsevich2017a} is formulated without $w$. However, the generalization required to account for it is trivial. This is why the above theorem is stated with $w$ in \eqref{final-lim-v1}.

We can assume that the coordinates are selected so that $\vec\Theta_0:=(1,0)$. Suppose $\s$ is parametrized by $[-a,a]\ni\theta\to y_*(\theta)\in\s$. Here $y_*(\theta)$ is the point where the line $\{x\in\br^2:(x-y_*(\theta))\cdot\vec\Theta=0\}$ is tangent to $\s$. By assumption, $R(\theta)=\vec\Theta\cdot y_*''(\theta)>0$, $|\theta|\le a$, where $R(\theta)$ is the radius of curvature of $\s$ at $y_*(\theta)$. If $x_0\in\s$, we assume $x_0=y_*(0)$.

Let $H_\e(s)$, $s\in\br$, be a family of functions defined for all $\e>0$ sufficiently small, with the following properties
\begin{itemize}
\item[$H1$.] There exists $c$ such that $|\e^{-1}H_\e(s)|\le c$ for all $s\in\br$ and all $\e>0$ sufficiently small;
\item[$H2$.] $\e^{-1}H_\e(\e^{1/2}s)$ is uniformly Holder continuous with exponent $\ga$, $0<\ga\le 1$, i.e. 
\be\label{holder}
\sup_{s\in\br,h>0,\e>0}\frac{|H_\e(s+\e^{1/2}h)-H_\e(s)|}{\e h^\ga}<\infty.
\ee
\end{itemize}
It is convenient to introduce the normalized function
\be\label{H0-defs}
H_0(s):=\e^{-1}H_\e(\e^{1/2}s).
\ee
The dependence of $H_0$ on $\e$ is omitted from notation for simplicity. Define also
\be\label{main-fn}
f_\e(x)=\begin{cases}
\Delta f(x),& 0 < t<H_\e(\theta),\\
-\Delta f,& H_\e(\theta)<t < 0,\\
0,& \text{in all other cases},
\end{cases}
\ \Delta f(x):=f_+(x)-f_-(x),\ x=y_*(\theta)+t\vec\Theta.
\ee
As is easily seen, $f_\e^{mod}(x):=f(x)-f_\e(x)$ is a function, in which $\s$ is modified by $H_\e$, see Figure~\ref{fig:perturbation}. At the points where $H_\e(\theta)>0$, a small region is removed from $D_+$ and added to $D_-$. At the points where $H_\e(\theta)<0$, a small region is removed from $D_-$ and added to $D_+$. The magnitude of the perturbation is $O(\e)$. Let $\s_\e$ denote the perturbed boundary. Thus, $f_\e^{mod}(x)$ is discontinuous across $\s_\e$ instead of $\s$.

\begin{figure}[h]
{\centerline
{\epsfig{file=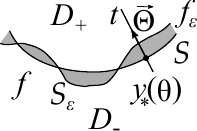, width=4.5cm}}
}
\caption{Illustration of the perturbation $\s\to\s_\e$ and the function $f_\e$, which is supported in the shaded region.}
\label{fig:perturbation}
\end{figure}

In Theorem~\ref{lem-phi-lim} we obtained the DTB in the case of a sufficiently smooth $\s$. By linearity, we can ignore the original function $f$ and consider the reconstruction of only the perturbation $f_\e$. 

Introduce the notation
\be\label{H-def}\begin{split}
\chi_H(t):=&\begin{cases}1,&0\le t\le H,\\
0,&t\not\in [0,H],\end{cases} \text{ if $H>0$, and } \chi_H(t):=\begin{cases}-1,&H\le t\le 0,\\
0,&t\not\in [H,0],\end{cases} \text{ if $H<0$}.
\end{split}
\ee
Let $\frec$ denote the reconstrution of only the perturbation $f_\e$ \eqref{main-fn}. The first result in this paper is as follows:

\begin{theorem}\label{main-res} Let (a) $f$ satisfy conditions f1--f3; (b) interpolation kernel $\ik$ satisfy conditions IK1, IK2; (c) perturbation $H_\e$ satisfy conditions H1, H2; and (d) detector aperture function $w$ satisfy $w\in C_0^2(\br)$ and $\int w(p)dp=1$. Suppose $x_0\in\s$ is generic, and let $\vec\dir_0$ be the positive unit normal to $\s$ at $x_0$.  If $\frec$ is the reconstruction of $f_\e$ from the data \eqref{data_eps} (with $f=f_\e$) using \eqref{recon-orig}, then
\be\label{final-lim-v2}
\lim_{\e\to0}\left[\frec(x_0+\e\hat x)- \Delta f(x_0)(\ik*w*\chi_{H_0(0)})(\vec\Theta_0\cdot\hat x)\right]=0,
\ee
where $\Delta f$ is the same as in \eqref{main-fn}. If $x_0\not\in\s$ is generic, then 
\be\label{final-lim far zone v2}
\lim_{\e\to0}\frec(x_0+\e\hat x)= 0.
\ee
\end{theorem}

Comparing and combining \eqref{final-lim-v1} and \eqref{final-lim-v2} we see that a non-smooth perturbation leads to the following two effects:
\begin{enumerate}
\item Local shifting of the reconstructed boundary $x_0\to x_0+H_\e(\theta_0)\vec\Theta_0$, 
and
\item The DTB retains its structure of the convolution of the ideal edge response (step function with the jump at $x_0+H_\e(\theta)\vec\Theta$) with $\ik*w$.
\end{enumerate} 

\section{Beginning of the proof of Theorem~\ref{main-res}}\label{sec:beg proof}

Pick a generic $x_0\in\s$. By linearity, in what follows we can consider only one domain $U\ni x_0$, and make the following assumptions: 
\begin{enumerate}
\item $\text{supp}(f)\subset U$,  
\item $\s$ is sufficiently short;
\item $f\equiv0$ in a neighborhood of the endpoints of $\s$.
\end{enumerate}
By assumption (3) above, $f(x)\equiv0$ in a neighborhood of $y_*(\pm a)$. Using assumption f2, that $x_0$ is generic, and $\s$ is sufficiently short (i.e., $0<a\ll 1$) we have:
\begin{enumerate}
\item $y_*(\theta)$ is a regular parametrization with $y_*\in C_b^4([-a,a])$ (i.e., bounded derivatives up to the fourth order);
\item $\s$ satisfies
\be\label{short-cond}
\frac{\max_{|\theta|\le a}|\vec\Theta\cdot y_*'''(\theta)|}{\min_{|\theta|\le a}R(\theta)} a \ll 1;
\ee
\item There exists $c>0$ such that 
\be\label{aldot-ineq}
\vec\Theta\cdot (x_0-y_*(\theta))\ge c\theta^2,\ |\theta|\le a; 
\ee
\item No line $\{x\in\br^2:\vec\al\cdot x=p\}$ is tangent to $\s$ if $a< |\al|\le \pi/2$. 
\end{enumerate}
Additional requirements on the smallness of $a$ will be formulated later as needed.

Pick some $A\gg1$ and define 
\be\label{two-sets}
\Omega_1:=\{\al:\,|\al|\le A\e^{1/2}\},\ \Omega_2:=[-a,a]\setminus\Omega_1,\ \Omega_3:=[-\pi/2,\pi/2]\setminus [-a,a].
\ee
Let $f^{(j)}$ denote the reconstruction obtained by the formula in \eqref{recon-orig} with $\al_k$ restricted to $\Omega_j$, $j=1,2,3$. The first term is split into two: $f^{(1)}=f_1^{(1)}+f_2^{(1)}$, where $f_1^{(1)}$ is the leading singular term of $f^{(1)}$. The term $f_1^{(1)}$ is defined following equation \eqref{gpm-fn-lt} below.

The result in Theorem~\ref{main-res} involves the limit of $\frec(x_0+\e\hat x)$ as $\e\to0$. Of the three functions $f^{(j)}$ that make up $\frec$, two depend on the new parameter $A\gg1$. The logic of our proof is based on computing the double limit $\lim_{A\to\infty}\lim_{\e\to0}f^{(j)}$, $j=1,2$. It is important that $A$ is fixed when the limit as $\e\to 0$ is computed.

\subsection{Estimation of $f_1^{(1)}$.}\label{ssec:f11}
Substitute \eqref{main-fn} into \eqref{data_eps}
\be\label{g-fn}\begin{split}
\hat f_\e(\al,p)=&\frac1\e\int_{-a}^a\int_{0}^{H_\e(\theta)} w\left(\frac{p-\vec\al\cdot (y_*(\theta)+t\vec\Theta\,)}\e\right) F(\theta,t) \dd t \dd\theta \\
=&\int_{-a}^a\int_0^{\e^{-1}H_\e(\theta)} w\left(\hat P-\vec\al\cdot \frac{y_*(\theta)-y_*(\al)}\e-\hat t\cos(\theta-\al)\,\right)F(\theta,\e\hat t) \dd\hat t \dd\theta,\\
F(\theta,t):=& \Delta f(y_*(\theta)+t\vec\Theta) (R(\theta)-t),\ \hat P:=\frac{p-\vec\al\cdot y_*(\al)}\e.
\end{split}
\ee
Here $R(\theta)-t=\text{det}(\dd y/\dd (\theta,t))$.
We assume that $\e$ is sufficiently small, and $R(\theta)-t>0$ on the domain of integration.
Denote
\be\label{two-defs}
\tilde\al:=\al/\e^{1/2},\ \auxang:=\theta-\al,\ \tth:=\auxang/\e^{1/2}.
\ee
Generally, throughout the paper a hat above a variable denotes rescaling of the original variable by a factor of $\e$, and a tilde above a variable denotes rescaling by a factor of $\e^{1/2}$. For example, $\hat p=p/\e$ and $\tilde\al=\al/\e^{1/2}$. Then
\be\label{g-fn-v2}\begin{split}
\hat f_\e(\al,p)=&\e^{1/2}\gc(\al,\hat P), \\
\gc(\al,\hat p)
:=&\int_{\tth_-}^{\tth_+}\int_0^{\e^{-1}H_\e(\al+\e^{1/2}\tth)} w\left(\hat p-\vec\al\cdot \frac{y_*(\al+\e^{1/2}\tth)-y_*(\al)}\e-\hat t\cos(\e^{1/2}\tth)\,\right)\\
&\hspace{1cm}\times F(\al+\e^{1/2}\tth,\e\hat t) \dd\hat t \dd\tth\\
=&\int_{\tth_-}^{\tth_+}\int_0^{H_0(\tilde\al+\tth)} w\left(\hat p-\left(\frac{\vec\al\cdot y_*''(\al)\tth^2}2+\frac{O(\e^{3/2}|\tth|^3)}\e\right)-\hat t\left(1+O(\e\tth^2)\right)\right)\\
&\hspace{1cm}\times \left(\Delta f(y_*(\al))R(\al)+O(\e^{1/2}|\tth|)+O(\e)\right) \dd\hat t \dd\tth,
\end{split}
\ee
where $\tth_{\pm}=(\pm a-\al)/\e^{1/2}$. In this subsection, $|\al|=O(\e^{1/2})$, so $R(\al)-R(0)=O(\e^{1/2})$, and we can replace $\vec\al\cdot y_*''(\al)=R(\al)$ with $R_0=R(0)$ in the argument of $w$. The corresponding error term is $O(\e^{1/2}\tth^2)$. A similar argument holds for the product $\Delta f(y_*(\al))R(\al)$, and the corresponding error term is $O(\e^{1/2})$. 

Let us look at the leading term of $\gc$, which is obtained by neglecting all the big-$O$ terms and extending the integral with respect to $\tth$ to all of $\br$: 
\be\label{gpm-fn-lt}
\gl(\al,\hat p):=\Delta f(x_0)R_0\int_\br\int_0^{H_0(\tilde\al+\tth)} w\left(\hat p-(R_0/2) \tth^2-\hat t\right) \dd\hat t \dd\tth,
\ee
where $\tilde\al$ and $H_0$ are the same as in \eqref{two-defs} and \eqref{H0-defs}, respectively. 
The subscript `$l$' signifies that $\gl$ is the leading term of $\gc$. By definition, substitution of the resulting approximation $\hat f_\e(\al,p)\approx \e^{1/2}g_l(\al,\hat P)$ into \eqref{recon-orig} (with $\hat P$ as in \eqref{g-fn}) and restricting the sum to $\al_k\in\Omega_1$ gives $f_1^{(1)}$.

\begin{lemma}\label{lem:g1-props} Fix any $\de$, $0<\de<1/2$. For some $c$ and any $\al\in(-a,a)$ one has
\be\label{left-supp-pm}
g_*(\al,\hat p)\equiv 0 \text{ for }\hat p<c,
\ee 
\be\label{g-as-st2-pm}
g_*(\al,\hat p)=O(\hat p^{-1/2}),\ \hat p\to+\infty,
\ee
and
\be\label{delg-bnd-st2-pm}\begin{split}
&g_*(\al,\hat p+\Delta\hat p)-g_*(\al,\hat p)=
O\left(|\Delta\hat p|^\ga \hat p^{-(1+\ga)/2}\right) \text{ if }
|\Delta\hat p|=O(\hat p^{\de}),\, \hat p\to+\infty.
\end{split}
\ee
The above assertions hold for both $g_*=\gc$ and $g_*=\gl$. The big-$O$ terms in \eqref{g-as-st2-pm} and \eqref{delg-bnd-st2-pm} are uniform with respect to $\al\in (-a,a)$ and $\e>0$ sufficiently small.
\end{lemma}

A proof of the lemma is in Appendix~\ref{sec:lemI}. 
Substitute the data \eqref{data_eps} into the inversion formula \eqref{recon-orig} and use \eqref{g-fn}, \eqref{g-fn-v2}:
\be\label{ker-der}
\begin{split}
\frac1\pi\sum_{j} \int \frac{\pa_p\ik\left(\frac{p-p_j}\e\right)}{p-\al\cdot x}dp \hat f_\e(\al,p_j)
&=\e^{-1/2}\sum_{j} (\CH\ik')\left(\frac{\al\cdot x}\e-j\right)g\left(\al,j-\frac{\al\cdot y}{\e}\right),
\end{split}
\ee
where $\CH$ is the Hilbert transform. In this subsection we approximate $g\approx g_l$ and introduce:
\be\label{frst-sub}\begin{split}
\Psil(\tilde\al,\hat p,q):=&\sum_j (\CH\ik')\left(\hat p-j\right)\gl(\al,j-q).
\end{split}
\ee
Recall that we continue using the convention that $\al$ and $\tilde\al$ are related as in \eqref{two-defs}.


Even though $\Psil$ depends on $\e$ via $H_0$, this dependence is omitted from notation for simplicity. All the properties of $\Psil$ to be established below are uniform with respect to $\e$ and $\tilde\al$. Clearly, $\Psil(\tilde\al,\hat p-n,q-n)=\Psil(\tilde\al,\hat p,q)$ for any $n\in\mathbb Z$. 

We need the asymptotics of $\Psil(\tilde\al,\hat p,q)$ as $\hat p-q\to\infty$. By the invariance of $\Psil$ with respect to integer shifts, we can assume that $q$ is confined to a bounded set, e.g. $q\in[0,1)$, and $\hat p\to\infty$. The following result is proven in Appendix~\ref{sec:prf-lemPsi}.

\begin{lemma}\label{lem:Psi-asympt} One has 
\be\label{Psi-asymp}\begin{split}
\Psil(\tilde\al,\hat p,q)=\begin{cases} O(|\hat p|^{-3/2}),& \hat p\to-\infty,\\ 
O(\hat p^{-(1+\de)/2}),& \hat p\to+\infty,\end{cases} \
\de=\frac{\ga}{\ga+1},\ |\tilde\al| < \e^{-1/2}a,\ q\in[0,1),
\end{split}
\ee
where the big-$O$ terms are uniform with respect to $\tilde\al,q$ in the indicated sets and $\e>0$ sufficiently small. 
\end{lemma}

From \eqref{g-fn}, \eqref{g-fn-v2}, \eqref{gpm-fn-lt}, \eqref{ker-der}, and \eqref{frst-sub}, the leading term of the reconstruction is given by
\be\label{recon-1}
\begin{split}
f_1^{(1)}(x_0+\e\hat x)=&-\frac1{2\pi}\frac{\Delta\al}{\e^{1/2}}\sum_{\al_k\in\Omega_1} \Psil\left(\frac{\al_k}{\e^{1/2}},\frac{\vec\al_k\cdot x}\e,\frac{\vec\al_k\cdot y_*(\al_k)}\e\right)\\
=&-\frac{\kappa\e^{1/2}}{2\pi} \sum_{\al_k \in\Omega_1} \Psil\left(\frac{\al_k}{\e^{1/2}},\vec\al_k\cdot \hat x+\frac{\vec\al_k\cdot (x_0-y_*(\al_k))}\e+q_k,q_k\right),\\
q_k:=&\left\{\frac{\vec\al_k\cdot y_*(\al_k)}\e\right\}.
\end{split}
\ee
Here $\{r\}:=r-\lfloor r\rfloor$ denotes the fractional part of $r$, and $\lfloor r\rfloor$ is the floor function, i.e. the largest integer not exceeding $r$.

\begin{lemma}\label{lem:Psi-cont} One has:
\be\begin{split}\label{Psi-cont}
\Psil(\tilde\al+\Delta \tilde\al,\hat p,q)-\Psil(\tilde\al,\hat p,q)=&O(|\Delta \tilde\al|^{\ga}),\ \Delta \tilde\al\to0,\\\Psil(\tilde\al,\hat p+\Delta \hat p,q)-\Psil(\tilde\al,\hat p,q)=&O(|\Delta \hat p|^\mu),\ \Delta \hat p\to0,\\
\Psil(\tilde\al,\hat p,q+\Delta q)-\Psil(\tilde\al,\hat p,q)=&O(|\Delta q|),\ \Delta q\to0,
\end{split}
\ee
for any $\mu<1$. The big-$O$ terms are uniform with respect to $\tilde\al\in \e^{-1/2}(-a,a)$, $\hat p$ and $q$ confined to any bounded set, and $\e>0$ sufficiently small.
\end{lemma}

\begin{proof} The first line follows from \eqref{holder}, \eqref{gpm-fn-lt}, \eqref{frst-sub}, and the fact that $\CH\ik'(t)=O(t^{-2})$, $t\to\infty$. 

To prove the second line, note that the pseudo-differential operator $\CH \pa/\pa t\in S_{1,0}^1(\br\times\br)$ and $\ik\in C_0^2(\br)$, so $\CH \ik'\in C_*^1(\br)$, where $C_*^s(\br)$, $s>0$, denotes the Holder-Zygmund space (see item 2 in Remark 6.4 and Theorem 6.19 in \cite{Abels12}). Since $C_*^1(\br)\subset C_*^\mu(\br)$ for any $\mu\in(0,1)$, and the latter space consists of functions that are H{o}lder continuous with exponent $\mu$ (e.g., see Theorem 6.1 in \cite{Abels12}), we get that
\be\label{hphi}
|\CH\ik'(t+\Delta t)-\CH\ik'(t)|\le c_1\begin{cases} |\Delta t|/(1+t^2),& |t|\ge c_2,\\
|\Delta t|^{\mu},\ |t|\le c_2,\end{cases}
\ee
for some $c_{1,2}>0$ and all $|\Delta t|$ sufficiently small. The desired assertion now follows.

The third line follows by replacing $q$ with $q+\Delta q$ in \eqref{frst-sub} and then in \eqref{gpm-fn-lt}, 
subtracting $\Psil(\tilde\al,\hat p,q)$ from $\Psil(\tilde\al,\hat p,q+\Delta q)$, and then using that $w$ is compactly supported. 
\end{proof}

By \eqref{recon-1} and the second line in Lemma~\ref{lem:Psi-cont},
\be\label{partI-two}
f_1^{(1)}(x_0+\e\hat x)
=-\frac{\kappa\e^{1/2}}{2\pi} \sum_{\al_k \in\Omega_1} \Psil\left(\frac{\al_k}{\e^{1/2}},\vec\Theta_0\cdot \hat x+\frac{R_0}2 \frac{\al_k^2}\e+q_k,q_k\right)+O(\e^{\mu/2}).
\ee
We can use Lemma~\ref{lem:Psi-cont}, because the arguments of $\Psi_l$ remain bounded when $\al_k \in\Omega_1$. With $\tilde\al_k:=\al_k/\e^{1/2}$, we have $\Delta \tilde\al=\kappa\e^{1/2}$. Using that $x_0$ is generic and following the same approach as in \cite{Katsevich2017a} leads to (Condition (2) is essential in this step) 
\be\label{f11-lim}
\lim_{\e\to0}\left(f_1^{(1)}(x_0+\e\hat x)
+\frac1{2\pi} \int_{|\tilde\al|\le A}\int_0^1 \Psil\left(\tilde\al,\vec\Theta_0\cdot \hat x+\frac{R_0}2 \tilde\al^2+q,q\right)\dd q \dd\tilde\al\right)=0.
\ee
The double integral in \eqref{f11-lim} does not necessarily have a limit as $\e\to0$, since the dependence of $H_0$ and $\Psi_l$ on $\e$ can be complicated.

Note that the argument in this subsection establishes the limit in \eqref{f11-lim}, but it does not say anything about the rate of convergence.

\subsection{Estimation of $f_2^{(1)}$.}
Define (cf. \eqref{g-fn-v2} and \eqref{gpm-fn-lt}): 
\be\label{delg-fn}
\Delta g(\al,\hat p):=\gc(\al,\hat p)-\gl(\al,\hat p).
\ee
Note that $\gl$ (and, therefore, $\Delta g$) is not compactly supported in $\hat p$, even though $\gc$ is compactly supported. However, inserting in \eqref{gpm-fn-lt} the same limits $\tth_{\pm}$ as in \eqref{g-fn-v2}, introduces only a small error: 
\be\label{delgl}
\gl(\al,\hat p)-\Delta f(x_0)R_0\int_{\tth_-}^{\tth_+}\int_0^{H_0(\tilde\al+\tth)} w\left(\hat p-(R_0/2) \tth^2-\hat t\right) \dd\hat t \dd\tth=O(\e^{1/2}).
\ee
The big-$O$ term in \eqref{delgl} is uniform with respect to $\al$ in compact subsets of $(-a,a)$. The latter condition ensures that $\tth_{\pm}=O(\e^{-1/2})$. In this subsection we assume that $\al\in\Omega_1$,  so the estimate \eqref{delgl} is indeed uniform. Since $w\in C_0^2(\br)$ and $F\in C^2([-a,a]\times[-\de,\de])$ for some $\de>0$, we have from \eqref{g-fn-v2}, the paragraph following \eqref{g-fn-v2}, and \eqref{delgl}
\be\label{delg-fn-bnd}\begin{split}
\left|\Delta g(\al,\hat p)\right|
&\le  c\begin{cases}
|\hat p|^{-1/2}\left[\e^{1/2}|\hat p|+ \e^{-1}(\e|\hat p|)^{3/2}\right],& c\le|\hat p|\le c/\e\\
\e^{1/2},& |\hat p|\le c \text{ or }|\hat p|\ge c/\e
\end{cases}\\
&\le c\e^{1/2}(|\hat p|+1)
\end{split}
\ee
for some $c>0$, which may have different values in different places. Here we have used that in \eqref{g-fn-v2}, $c_1\le \tilde\nu^2/|\hat p|\le c_2$ if $|\hat p|\ge c_3$ for some $c_{1,2,3}>0$.

Define similarly to \eqref{frst-sub}:
\be\label{frst-sub-del}
\Delta \Psi(\tilde\al,\hat p,q):=\sum_j (\CH\ik')\left(\hat p-j\right)\Delta g(\al,j-q).
\ee
Substitute \eqref{delg-fn-bnd} into \eqref{frst-sub-del}:
\be\label{del-psi1-est}
\begin{split}
|\Delta \Psi(\tilde\al,\hat p,q)|&\le \sum_{|j|\ge O(\e^{-1})} \frac{O(\e^{1/2})}{1+(\hat p-j)^2}+
\sum_{|j|\le O(\e^{-1})} \frac{O(\e^{1/2})(|j|+1)}{1+(\hat p-j)^2}\\
&=O(\e^{1/2}\ln(1/\e)),
\end{split}
\ee
assuming that $\hat p$ in \eqref{del-psi1-est} is confined to a bounded set, e.g. $|\hat p|\le c$. Here $c>0$ can be any fixed number. The restriction on $\hat p$ is justified, since $\al_k\in\Omega_1$. As is seen from \eqref{two-sets}, \eqref{recon-1}, and \eqref{partI-two}, $|\hat p|\lesssim (R_0/2) A^2$ when $A\gg 1$.
Clearly, the estimate \eqref{del-psi1-est} is uniform with respect to $\hat p\in [-c,c]$, $\al\in\Omega_1$ (i.e., $|\tilde\al|\le A$), $q\in [0,1)$, and all $\e>0$ sufficiently small. Recall that $A\gg1$ is fixed when computing the limit as $\e\to0$.
Summing over all $\al_k\in\Omega_1$ similarly to \eqref{recon-1}, yields
\be\label{recon-21}
f_2^{(1)}(x_0+\e\hat x)=\e^{1/2}\sum_{|\al_k|\in \Omega_1} O(\e^{1/2}\ln(1/\e))=O(\e^{1/2}\ln(1/\e)).
\ee

\subsection{Estimation of $f^{(2)}$.}
From \eqref{g-fn-v2}, we get similarly to \eqref{recon-1}:
\be\label{partII-est-v2}\begin{split}
f^{(2)}(x_0+\e \hat x)=&-\e^{1/2}\frac{\kappa}{2\pi} \sum_{\al_k \in\Omega_2}\Psi\left(\al_k,\vec\al_k\cdot \hat x+\frac{\vec\al_k\cdot (x_0-y_*(\al_k))}\e+q_k,q_k\right),\\
\Psi(\al,\hat p,q):=&\sum_j (\CH\ik')\left(\hat p-j\right)\gc(\al,j-q).
\end{split}
\ee
By Lemma~\ref{lem:g1-props}, $\Psi$ also satisfies \eqref{Psi-asymp}.
By \eqref{aldot-ineq},
\be\label{partII-est}\begin{split}
|f^{(2)}(x_0+\e \hat x)|\le&c_1\e^{1/2} \sum_{\al_k \in\Omega_2} \left|\Psi\left(\al_k,\vec\al_k\cdot \hat x+\frac{\vec\al_k\cdot (x_0-y_*(\al_k))}\e+q_k,q_k\right)\right|\\
\le &c_2 \e^{1/2} \sum_{k\ge A/\e^{1/2}}(k^2\e)^{-(1+\de)/2}
\le c_3\int_A^\infty x^{-(1+\de)} \dd x
= O(A^{-\de})
\end{split}
\ee
for some $c_{1,2,3}$. 
Consequently,
\be\label{partII-lim-final}
\lim_{A\to\infty}\lim_{\e\to0}f^{(2)}(x_0+\e \hat x)=0.
\ee

\subsection{Estimation of $f^{(3)}$.}\label{ssec:estf3}
Consider now the third term $f^{(3)}$. By construction, $\al\in\Omega_3$ implies that $\al\cdot y_*'(\theta)\not=0$, $|\theta|\le a$. Hence, we can express $\theta$ in terms of $s$ by solving $s=\al\cdot y_*(\theta)$. Suppose, for example, that $\al\cdot y_*'(\theta)>0$. The case when this expression is negative is completely analogous. From the first line in \eqref{g-fn} we find
\be\label{g-fn-sigma3}\begin{split}
\hat f_\e(\al,\e \hat p)=&\int_{\al\cdot y_*(-a)}^{\al\cdot y_*(a)}\int_0^{\e^{-1}H_\e(\theta(s))} w\left(\hat p-\frac s\e -\hat t\cos(\theta(s)-\al)\,\right)
F(\theta(s),\e\hat t) \dd\hat t\, \theta'(s)\dd s\\
=&\e\int_\br\int_0^{\e^{-1}H_\e(\theta(\e\hat s))} w\left(\hat p-\hat s-\hat t\cos(\theta(\e\hat s)-\al)\,\right)\dd\hat t\, F_1(\e\hat s)\dd\hat s+O(\e^2),\\
F_1(s):=&F(\theta(s),0)\theta'(s).
\end{split}
\ee

Strictly speaking, $\dd s/\dd\theta$ can approach zero (i.e., $\theta'(s)\to\infty$) when $\al\to a^+$ and $\theta\to a^-$ or $\al\to -a^-$ and $\theta\to -a^+$. However, $f(x)\equiv0$ in a neighborhood of $y_*(\pm a)$, so $\theta'(s)$ is bounded on the support of $F(\theta(s),\e\hat t)$. Additionally, this allows us to (i) extend $F_1(s)$ from $[\al\cdot y_*(-a),\al\cdot y_*(a)]$ to $\br$ by zero without reducing the smoothness of $F_1$, and (ii) integrate with respect to $\hat s$ over $\br$.

Using \eqref{holder}, that $w$ is compactly supported, and Property $H1$, replace $\hat s$ with $\hat p$ in the arguments of $\theta$ and $F_1$ in the last integral in \eqref{g-fn-sigma3}
\be\label{gfn-s3-st2}\begin{split}
\hat f_\e(\al,\e \hat p)=&\e\int\int_0^{\e^{-1}H_\e(\theta(\e\hat p))} w\left(\hat p-\hat s-\hat t\cos(\theta(\e\hat p)-\al)\,\right)\dd\hat t\, F_1(\e\hat p)\dd\hat s+O(\e^{1+(\ga/2)})\\
=&\e F_1(\e\hat p)H_0\left(\e^{-1/2}\theta(\e \hat p)\right)+O(\e^{1+(\ga/2)}).
\end{split}
\ee
Substitute \eqref{gfn-s3-st2} into \eqref{recon-orig} and sum over $j$
\be\label{gfn-sub}
\begin{split}
&\frac1{\e}\sum_{j} (\CH\ik')\left(\hat p-j\right)\left(\e F_1(\e j)H_0\left(\e^{-1/2}\theta(\e j)\right)+O(\e^{1+(\ga/2)})\right)=I_\e(\al,\hat p)+O(\e^{\ga/2}),\\
&I_\e(\al,\hat p):=\sum_{j} (\CH\ik')\left(\hat p-j\right)F_1(\e j)H_0\left(\e^{-1/2}\theta(\e j)\right).
\end{split}
\ee
Since $\sum_{j} (\CH\ik')\left(\hat p-j\right)\equiv0$, we have
\be\label{gfn-sub-v2}\begin{split}
I_\e(\al,\hat p)=\sum_{j} (\CH\ik')\left(\hat p-j\right)\left(F_1(\e j)H_0\bigl(\e^{-1/2}\theta(\e j)\bigr)
-F_1(\e \hat p)H_0\bigl(\e^{-1/2}\theta(\e \hat p)\bigr)\right),
\end{split}
\ee
and
\be\label{gfn-sub-est}
|I_\e(\al,\hat p)|\le c\sum_j (1+\left(\hat p-j\right)^2)^{-1}\left(|\e^{1/2}(j-\hat p)|^{\ga}+\frac{\e|j-\hat p|}{1+\e|j-\hat p|}\right)=O(\e^{\ga/2}).
\ee
The second term in parentheses is written in this form, because $F_1$ is bounded. Writing it in a more conventional form $\e|j-\hat p|$ would cause the series in the upper bound to diverge. 
The above estimate is uniform with respect to $\al\in\Omega_3$, $\hat p\in\br$, and $\e>0$ sufficiently small. Summing over $\al_k$ in the inversion formula we find
\be\label{partI-three}
f^{(3)}(x)
=c\Delta \al \sum_{\al_k \in\Omega_3} I_\e(\al_k,\vec\al_k\cdot x/\e)+O(\e^{\ga/2})=O(\e^{\ga/2}).
\ee
 
\section{Computing the DTB. Examples.}\label{sec: dtb ex}
Since $A\gg 1$ can be arbitrarily large, \eqref{f11-lim}, \eqref{recon-21}, \eqref{partII-lim-final}, and \eqref{partI-three} imply
\be\label{f1-lim}\begin{split}
&\lim_{\e\to0}\left(\frec(x_0+\e\hat x)
-I(\vec\Theta_0\cdot \hat x,\e)\right)=0,\\
&I(h,\e):=-\frac1{2\pi} \int_{\br}\int_0^1 \Psil\left(\tilde\al,h+(R_0/2) \tilde\al^2+q,q\right)\dd q \dd\tilde\al.
\end{split}
\ee
Recall that the dependence of $I(h,\e)$ on $\e$ comes from the dependence of $H_0$ on $\e$, and $H_0$ appears in the definition of $\Psil$. By Lemma~\ref{lem:Psi-asympt}, the double integral above is absolutely convergent. Therefore,
\be\label{I-v0}\begin{split}
I(h,\e):=&-\frac1{2\pi}\lim_{\e_1\to0^+}J(h,\e_1),\\ 
J(h,\e_1):=&\int_{\br}\int_0^1 \Psil\left(\tilde\al,h+(R_0/2) \tilde\al^2+q,q\right)e^{-\e_1\tilde\al^2}\dd q \dd\tilde\al.
\end{split}
\ee
By \eqref{gpm-fn-lt} and \eqref{frst-sub}, the double integral in \eqref{I-v0} transforms to the following expression
\be\label{I-v1}\begin{split}
J(h,\e_1)=&C \iint_{\br^2} (\CH\ik')\left(h+(R_0/2) \tilde\al^2-q\right)e^{-\e_1\tilde\al^2}\\
&\hspace{0.5cm}\times\int_{\br}\int_0^{H_0(\tilde\al+\tth)} w\left(q-(R_0/2) \tth^2-\hat t\right) \dd\hat t \dd\tth \dd q \dd\tilde\al,\ C:=\Delta f(x_0)R_0.
\end{split}
\ee
We inserted an exponential factor in \eqref{I-v0}, since the quadruple integral in \eqref{I-v1} would otherwise not be absolutely convergent. Simplifying and changing variables $u=\tilde\al+\tth$, $v=\tilde\al-\tth$,   gives
\be\label{I-v1-st2}\begin{split}
J(h,\e_1)=&C \iint_{\br^2} \int_0^{H_0(\tilde\al+\tth)}  (\CH\ik'*w)\left(h+(R_0/2)(\tilde\al^2-\tth^2)-\hat t\right)e^{-\e_1\tilde\al^2} \dd\hat t \dd\tth \dd \tilde\al\\
=&\frac{C}2 \iint_{\br^2} (\CH\ik'*w*\chi_{H_0(u)})\left(h+(R_0/2)uv\right)e^{-\e_1(u+v)^2/4} \dd u\dd v.
\end{split}
\ee
See \eqref{H-def} for the definition of $\chi_H$. 

Represent the integrand in terms of its Fourier transform and integrate with respect to $v$:
\be\label{I-v1-st3}\begin{split}
J(h,\e_1)=&-\frac{C}2\frac1{2\pi} \iiint_{\br^3}|\la|\tilde\ik(\la)\tilde w(\la)\tilde\chi_{H_0(u)}(\la)e^{-i\la\left(h+(R_0/2)uv\right)}\\
&\hspace{3cm}\times e^{-\e_1(u+v)^2/4} \dd u \dd v \dd\la\\
=&-\frac{C}{4\pi}\left(\frac{4\pi}{\e_1}\right)^{1/2} \iint_{\br^2}  |\la|\tilde\ik(\la)\tilde w(\la)\tilde\chi_{H_0(u)}(\la)\exp\left(i\frac{\la R_0 u^2}2\right)\\
&\hspace{3cm}\times  \exp\left(-\frac{(\la R_0)^2}{4\e_1}u^2\right) \dd u\, e^{-i\la h}\dd\la,
\end{split}
\ee
where tildas above functions denote the 1D Fourier transform:
\be
\tilde\ik(\la)=\int \ik(x)e^{i\la x}dx.
\ee
The double integral in \eqref{I-v1-st3} converges absolutely, since $\tilde\ik(\la)=O(\la^{-2})$, $\la\to\infty$.

Changing the variable $s=u/\e_1^{1/2}$, we get by dominated convergence:
\be\label{I-v1-st4}\begin{split}
\lim_{\e_1\to 0^+}J(h,\e_1)
=&-\frac{C}{(4\pi)^{1/2}} \iint_{\br^2} |\la|\tilde\ik(\la)\tilde w(\la)\tilde\chi_{H_0(0)}(\la)\\
&\hspace{3cm}\times  \exp\left(-\frac{(\la R_0)^2}{4}s^2\right) \dd s\, e^{-i\la h} \dd\la\\
=&-\frac{C}{R_0} \int_{\br}\tilde\ik(\la)\tilde w(\la)\tilde\chi_{H_0(0)}(\la)\, e^{-i\la h} \dd\la.
\end{split}
\ee
An integrable upper bound is $c|\la||\tilde\ik(\la)|\exp(-(\la R_0s)^2/4)\in L^1(\br^2)$ for some $c>0$. Recall that $\tilde\ik(\la)\in L^1(\br)$ due to IK1. Finally,
\be\label{I-v1-st5}\begin{split}
I(h,\e)=&\frac{C}{R_0}(\ik*w*\chi_{H_0(0)})(h)=\Delta f(x_0)(\ik*w*\chi_{H_0(0)})(h).
\end{split}
\ee

%
%
\vspace{2mm}

\noindent
{\bf Example 1}: Constant width layer. Suppose $H_0(\theta)\equiv H$ is a constant. In this case $H_\e(\theta)=\e H$, and \eqref{final-lim-v1} and \eqref{final-lim-v2} are consistent with each other. 
Indeed, let us return to the situation in the remark following \eqref{main-fn}, where $f_\e$ modifies the original function $f$. Application of \eqref{final-lim-v1} to the modified function $f-f_\e$ gives \eqref{final-lim-v1}, where $\vec\Theta_0\cdot\hat x$ is replaced with $\vec\Theta_0\cdot\hat x-H$. This is precisely what we get by subtracting \eqref{final-lim-v2} from \eqref{final-lim-v1}.
\vspace{2mm}

\noindent
{\bf Example 2}: Fractal boundary. 
Suppose that $H_\e$ is given by
\be\label{He-def}\begin{split}
&H_\e(s):=\e^{1-(\ga/2)}\sum_{n=n_0(\e)}^{\infty} r^{-\ga n}\phi(r^n s),\\
&0<\ga<1,\ r>1,\ \phi(0)=0,\ n_0(\e)=c-\lfloor(1/2)\log_r\e\rfloor,\ c\in\mathbb Z.
\end{split}
\ee
where $\phi\in C_*^\bt(\br)$, i.e. $\phi$ is bounded and Holder continuous with exponent $\bt$, $\ga<\bt < 1$.  By \eqref{H0-defs},
\be\label{H0-eq}\begin{split}
H_0(s):=&\e^{-(\ga/2)}\sum_{n=n_0(\e)}^{\infty} r^{-\ga n}\phi(r^n \e^{1/2}s)\\
=&\sum_{n=c}^{\infty} r^{-\ga (n+q_\e)}\phi(r^{n+q_\e}s),\ q_\e=\{(1/2)\log_r\e\}.
\end{split}
\ee
The function $H_0$ is a real Weierstrass-type function (see \cite{Borodich1999}), which is continuous everywhere, differentiable nowhere, and its graph is a curve whose fractal dimension exceeds one \cite{Borodich1999}. See also \cite{Baranski2015} for a slightly less general case, where $\phi$ is $\mathbb Z$-periodic. 
It is well-known that $H_0$ is bounded and Holder continuous with exponent $\ga$. From this, properties $H1$ and $H2$ follow immediately. Thus, our approach allows the analysis of reconstruction of functions with singularities along rough (e.g., fractal) curves.

\section{Remote singularities. End of the proof of Theorem~\ref{main-res}}\label{sec: rem sing}

In this section we pick $x_0\not\in\s$ and show that the reconstruction of $f_\e$ from discrete data does not create artifacts in a neighborhood of $x_0$ (i.e., there are no nonlocal artifacts there) as long as $x_0$ is generic. This will prove \eqref{final-lim far zone v2}, the last assertion of Theorem~\ref{main-res}. 

Without loss of generality, we may suppose that $x_0$ satisfies $(x_0-y_*(0))\cdot\vec\Theta_0=0$, but $x_0\not=y_*(0)$. We can still use Lemma~\ref{lem:Psi-asympt}, because it is independent of the reconstruction point. Now, the function $\vec\Theta\cdot(x_0-y_*(\theta))$ has a root of first order at $\theta=0$ (by condition (1) in the definition of a generic point), so
\be\label{parts_I-II-est}\begin{split}
&|f^{(1)}(x_0+\e \hat x)+f^{(2)}(x_0+\e \hat x)|\\
&\le c_1\e^{1/2} \sum_{|\al_k| \le a} \left|\Psi\left(\al_k,\vec\al_k\cdot \hat x+\frac{\vec\al_k\cdot (x_0-y_*(\al_k))}\e+q_k,q_k\right)\right|\\
&\le c_2 \e^{1/2} \sum_{k=1}^{O(\e^{-1})}k^{-(1+\de)/2}
=O(\e^{\de/2}).
\end{split}
\ee
Here $\de$ is the same as in \eqref{Psi-asymp}. In this argument we assume that $a>0$ is sufficiently small. How small $a$ should be depends on the distance $|x_0-y_*(0)|$. This distance depends only on the properties of $f$, e.g. the geometry of $\s$, and, therefore, is fixed for any given $f$. So a sufficiently small $a>0$ can be selected and then held fixed throughout the proof.

The estimate in \eqref{partI-three} does not depend on the location of $x_0$, so it applies in this case as well. Hence $\frec(x_0+\e \hat x)\to 0$ as $\e\to 0$.

Using the linearity of the Radon transform and a partition of unity type of argument finishes the proof of Theorem~\ref{main-res}. 

\section{Numerical experiments I}\label{sec numexp I}

Our first experiment is with an oscillatory perturbation. The perturbed boundary $\s_\e$ oscillates around a curve $\s$, which is the boundary of the disc centered at $x_c=(0.1,0.2)$ with radius $R=0.3$. The equation of $\s_\e$ is $r(\theta)=R+2\e\cos(0.71\theta/\e^{1/2})$ in polar coordinates with the origin at the center of the disc. Here 
$$
\e=\Delta p=1.2/(N_p-1),\ \Delta_\theta=\pi/N_\theta,\ N_\theta=N_p-1.
$$
Obviously, the perturbation satisfies \eqref{holder} with $\ga=1$. The phantoms (i.e., the density plots of $f_\e(x)$) with $N_p=501$ and $N_p=1001$ are shown in Figire~\ref{fig:eop}. 

We use the Keys interpolation kernel \cite{Keys1981, btu2003}
\be\label{keys}
\ik(t)=3B_3(t+2) - (B_2(t+2) + B_2(t + 1)),
\ee
where $B_n$ is the cardinal $B$-spline of degree $n$ supported on $[0, n+1]$. The kernel is a piecewise-cubic polynomial, and $\ik,\ik'$ are continuous, hence $\ik\in C_0^2(\br)$.

We use this opportunity to correct the typo in \cite[eq. (6.1)]{Katsevich2017a}. The interpolating kernel used in the numerical experiments reported in \cite{Katsevich2017a} was not the cubic $B$-spline, but the Keys kernel \eqref{keys}.

\begin{figure}[h]
{\centerline{
{\vbox{
{\hbox{
{\epsfig{file={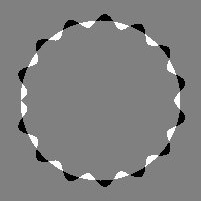}, width=4cm}}
{\epsfig{file={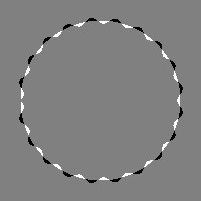}, width=4cm}}
}}
{\hbox{
{\epsfig{file={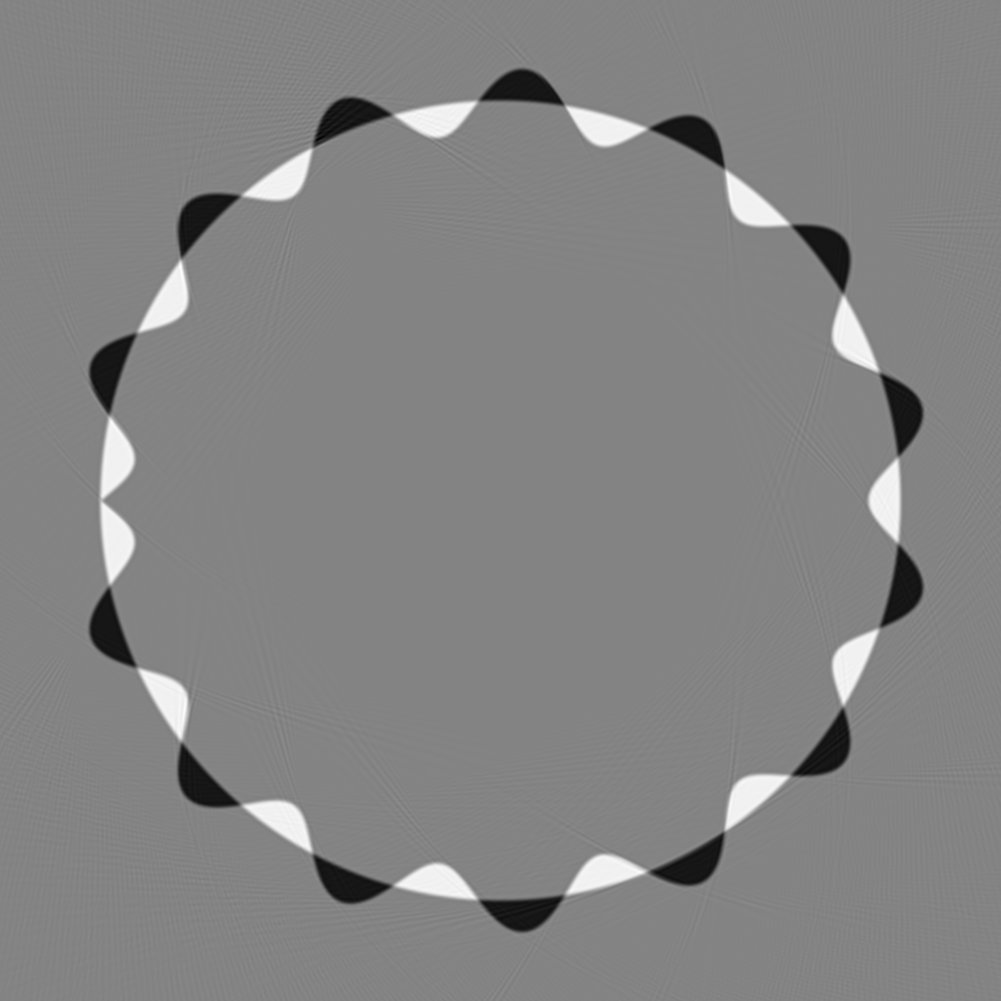}, width=4cm}}
{\epsfig{file={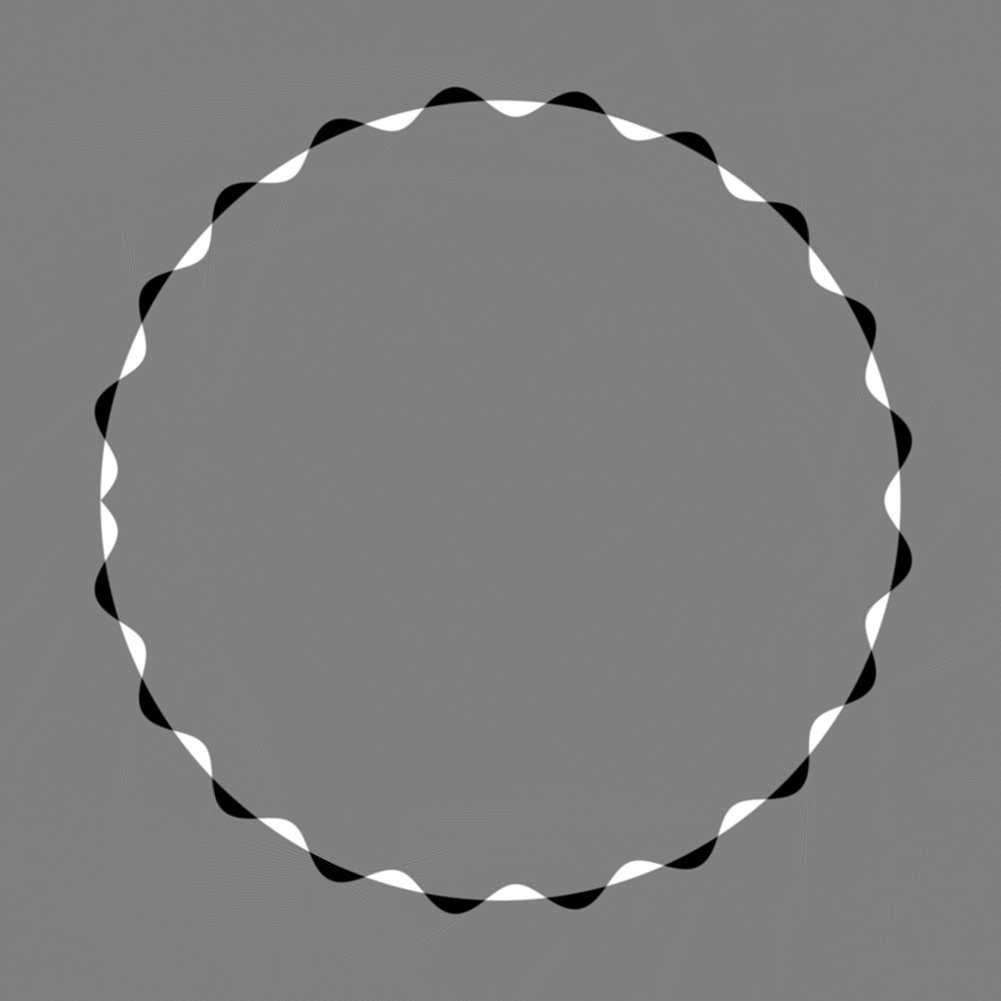}, width=4cm}}
}}
}}}}
\caption{Entire oscillatory phantom, $\ga=1$. Left column: $N_p=501$, right column: $N_p=1001$. Top row: originals, bottom row: recon\-structions.}
\label{fig:eop}
\end{figure}

Reconstructions of a region of interest (ROI) are shown in Figures~\ref{fig:ROI_op_500} and \ref{fig:ROI_op_1000}. The region of interest is centered at the point on the boundary of the disc
$$
x_0 = x_c - R(\cos(\al),\sin(\al)),\ \al=0.32\pi.
$$
The ROI is a square with side length $100\e$ centered at $x_0$. In Figure~\ref{fig:ROI_op_500}, $N_p=501$, and in Figure~\ref{fig:ROI_op_1000} - $N_p=1001$. In this and all other figures, grey color stands for pixel value 0, black color - pixel value (-1), and white color - pixel value 1. Notice that the ROI scales linearly with $\e$. Since the period of oscillations of $\s_\e$ scales like $\e^{1/2}$, the ROI contains fewer periods of $\s_\e$ as $\e$ decreases. Figures~\ref{fig:ROI_op_500} and \ref{fig:ROI_op_1000} demonstrate a good match between the DTBs (cf. Theorem~\ref{main-res}) and reconstructions.

\begin{figure}[h]
{\centerline{
{\hbox{
{\epsfig{file={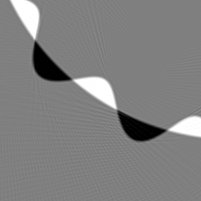}, width=2.5cm}}
{\epsfig{file={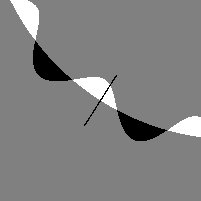}, width=2.5cm}}
{\epsfig{file={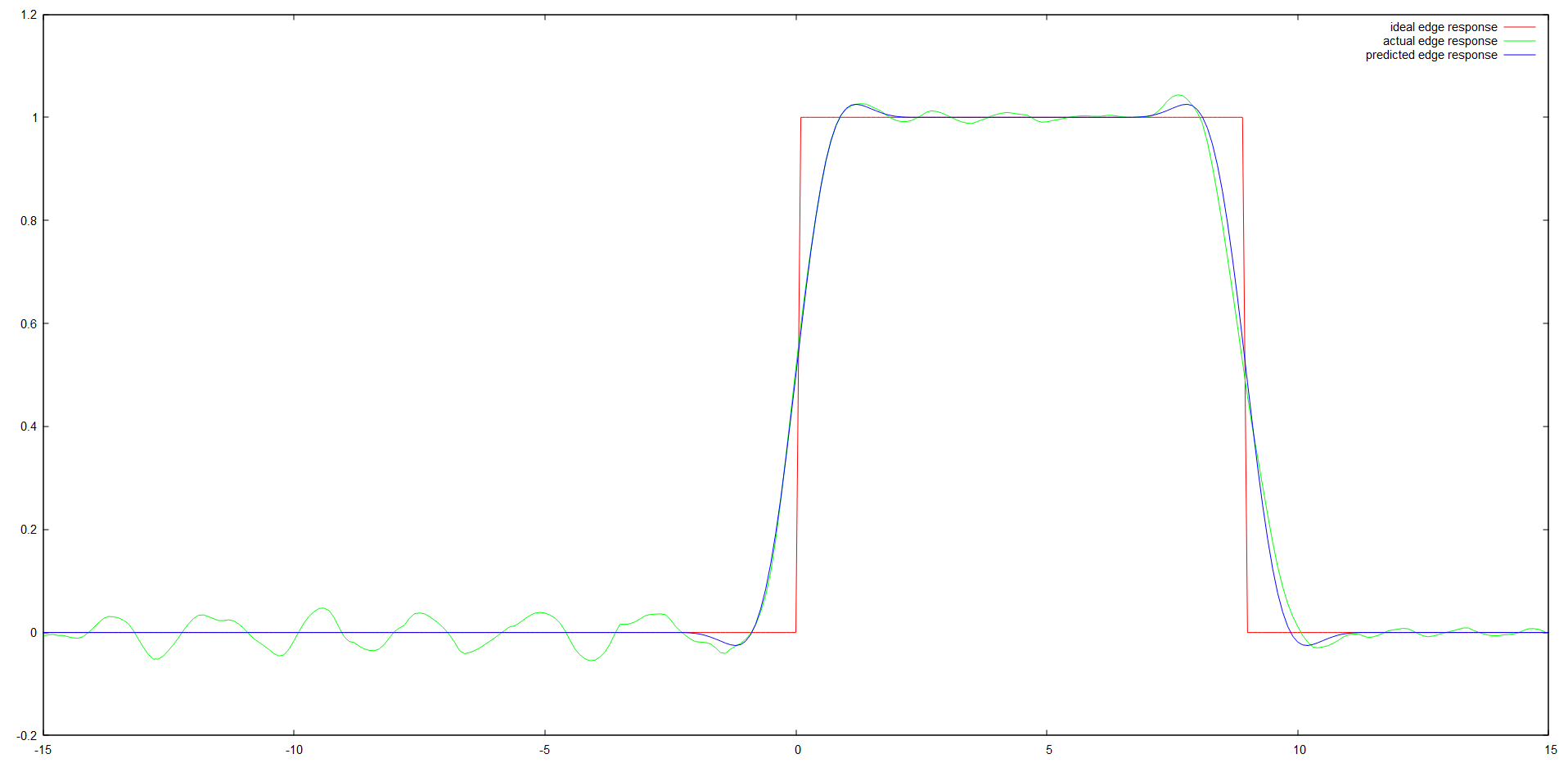}, width=6cm}}
}}}}
\caption{ROI in the oscillatory phantom, $\ga=1$, $\al=0.32\pi$, $N_p=501$. Left panel: reconstruction, middle panel: ground truth with the location of the profile shown, right panel: profiles along the line indicated in the middle panel.}
\label{fig:ROI_op_500}
\end{figure}

\begin{figure}[h]
{\centerline{
{\hbox{
{\epsfig{file={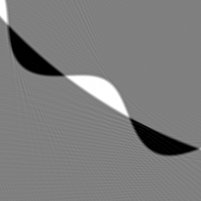}, width=2.5cm}}
{\epsfig{file={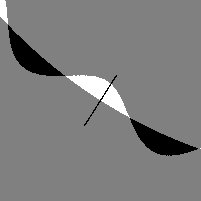}, width=2.5cm}}
{\epsfig{file={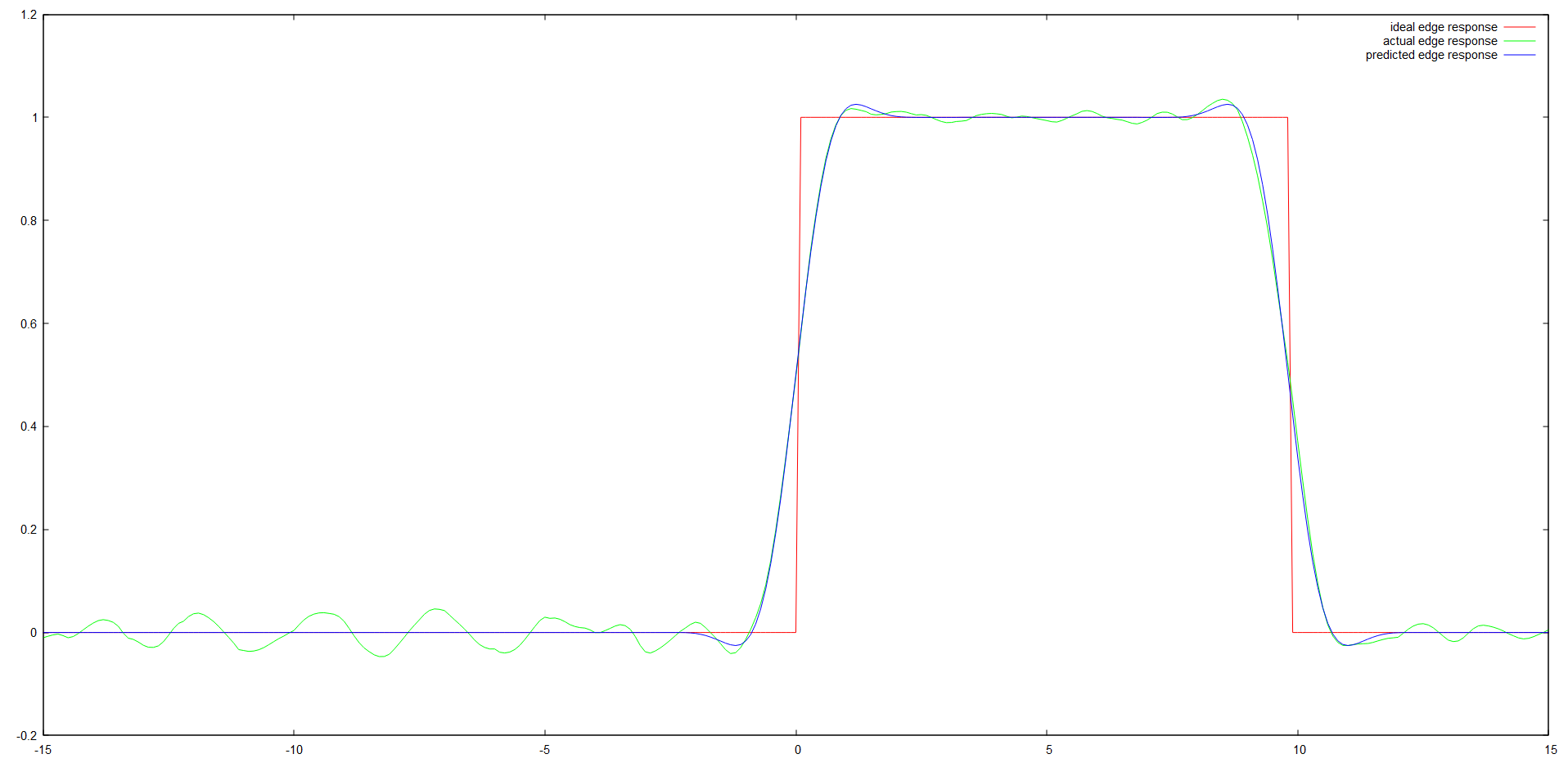}, width=6cm}}
}}}}
\caption{ROI in the oscillatory phantom, $\ga=1$, $\al=0.32\pi$, $N_p=1001$. Left panel: reconstruction, middle panel: ground truth with the location of the profile shown, right panel: profiles along the line indicated in the middle panel.}
\label{fig:ROI_op_1000}
\end{figure}

Our second experiment is with a phantom with a fractal perturbation. The fractal boundary $\s_\e$ is around the same disc as above. The equation of $\s_\e$ now is $r(\theta)=R+\e H_0(\theta/\e^{1/2})$ in polar coordinates with the origin at the center of the disc, where 
\be\label{H0 used}
H_0(s)=5\sum_{n=c}^{\infty} r^{-\ga n}\sin(r^n s),\ c=\lfloor\log_r(\pi)\rfloor,\ r=\sqrt{12},\ \ga=1/2.
\ee
The phantoms with $N_p=501$ and $N_p=1001$ are shown in Figire~\ref{fig:frp}. 

\begin{figure}[h]
{\centerline{
{\vbox{
{\hbox{
{\epsfig{file={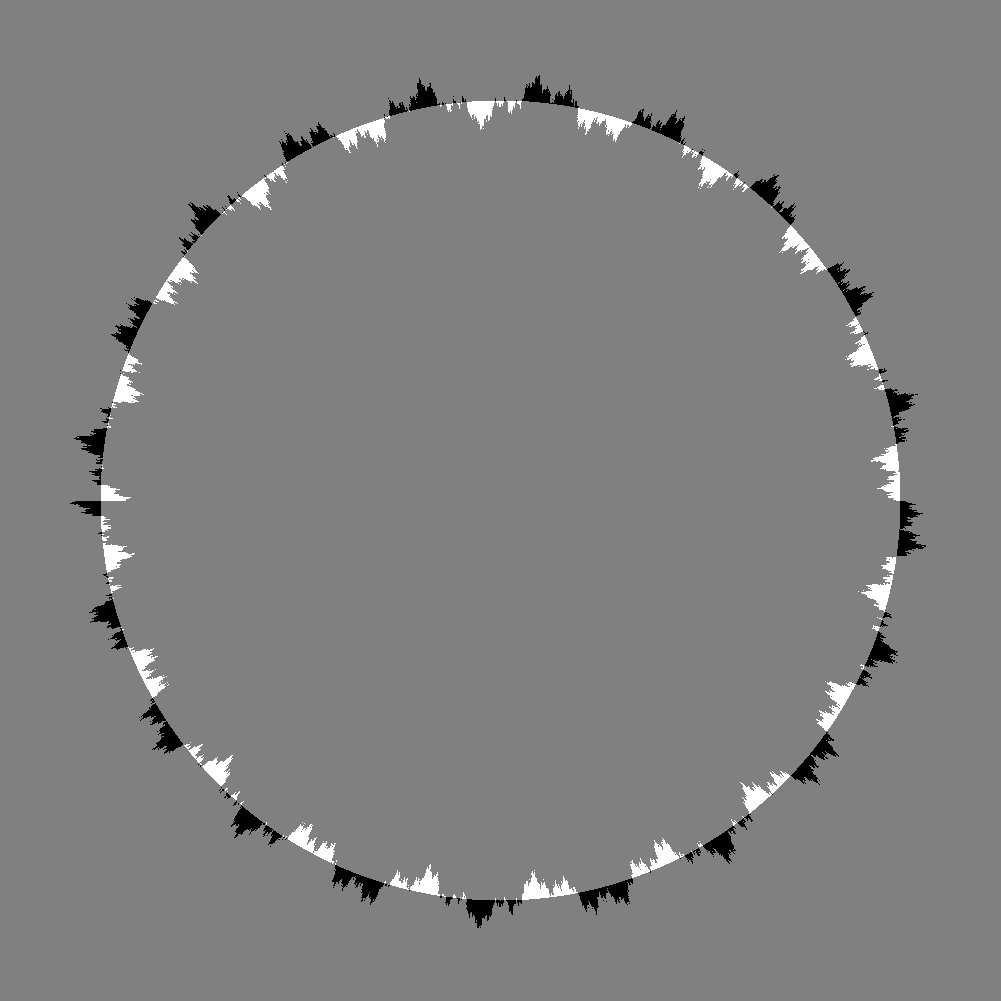}, width=4cm}}
{\epsfig{file={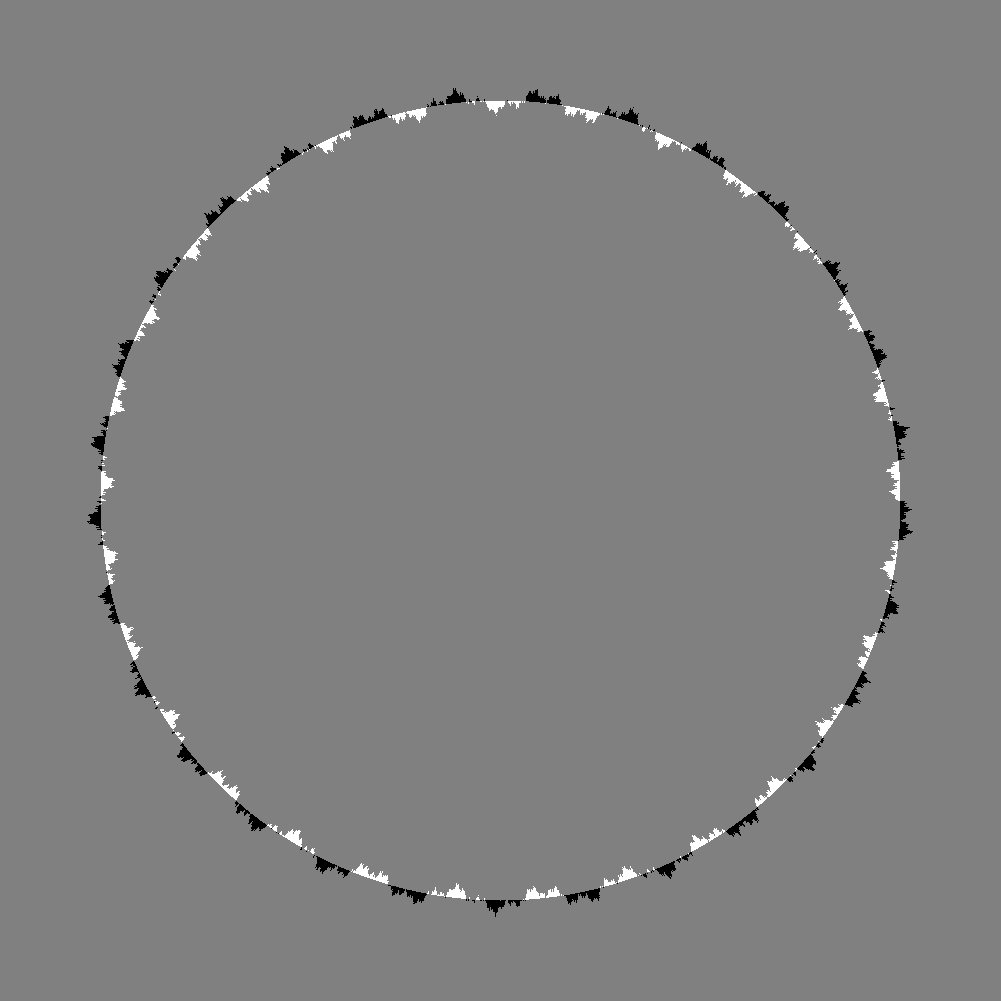}, width=4cm}}
}}
{\hbox{
{\epsfig{file={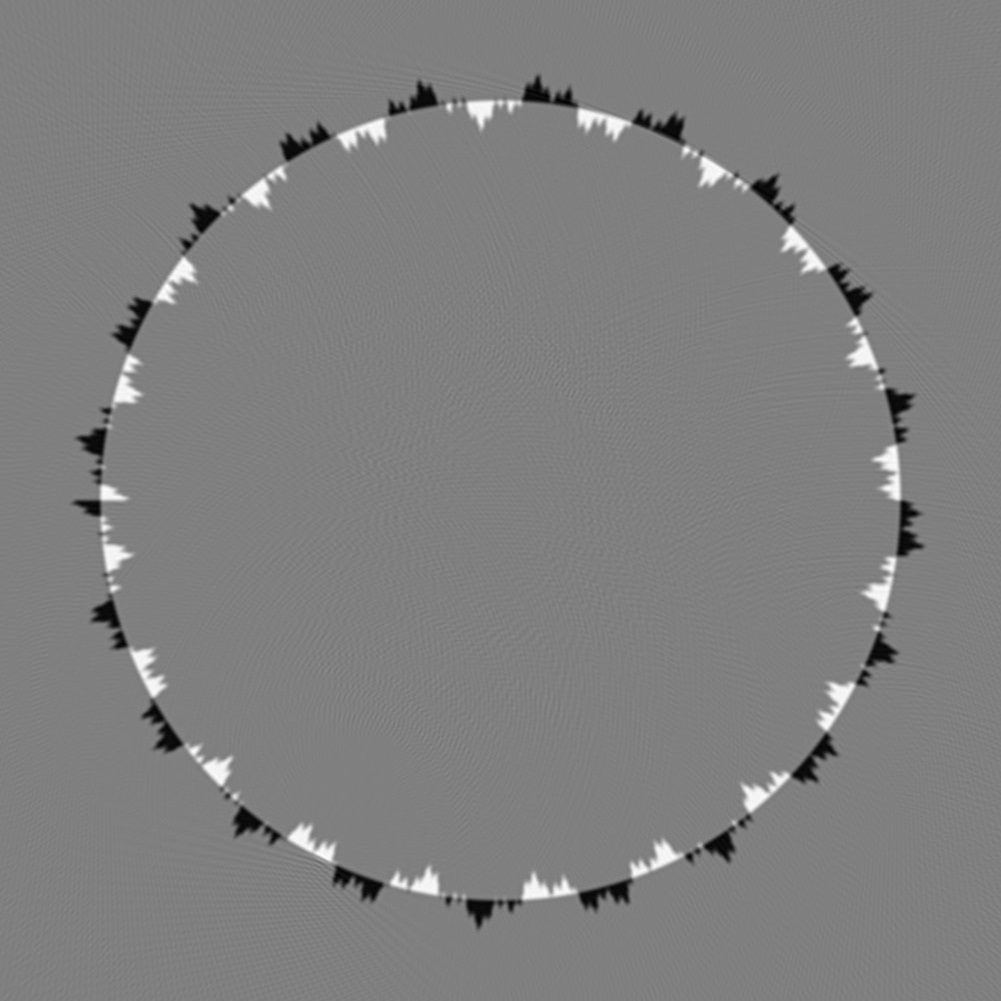}, width=4cm}}
{\epsfig{file={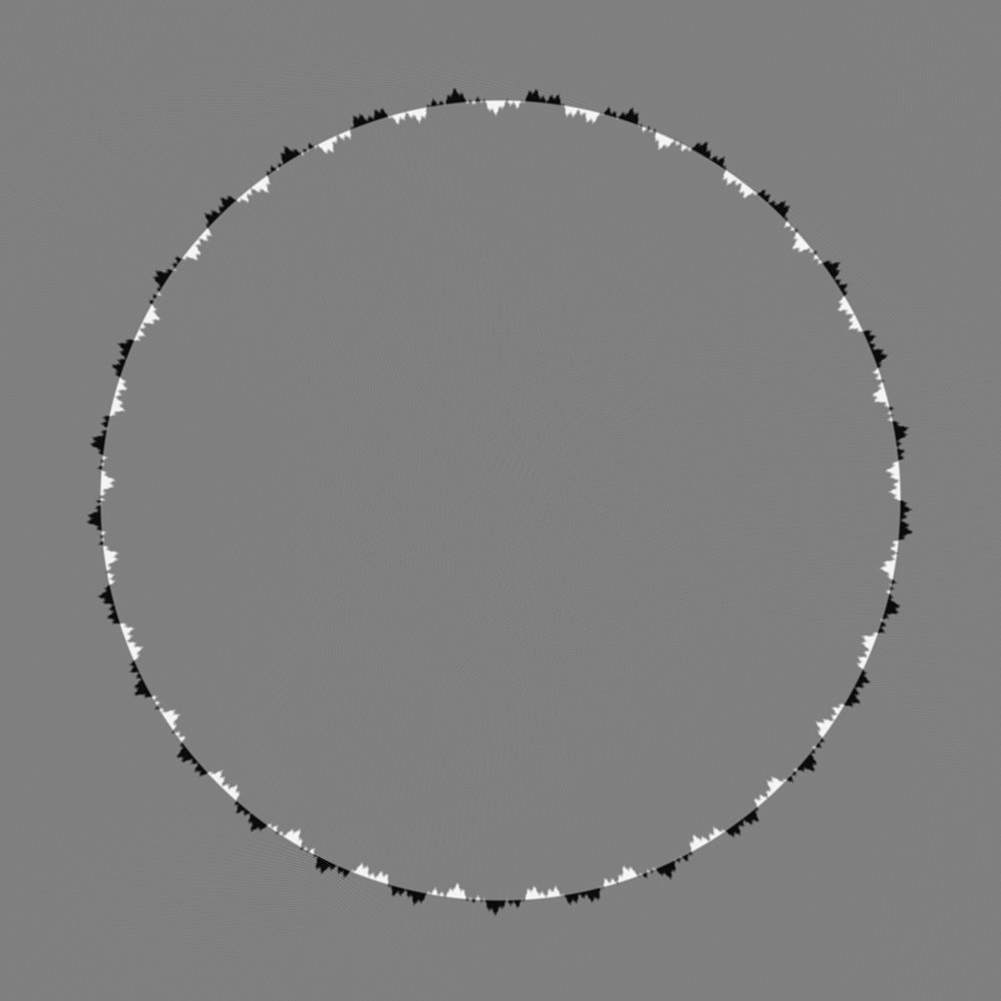}, width=4cm}}
}}
}}}}
\caption{Phantom with a fractal boundary, $\ga=1/2$. Left column: $N_p=501$, right column: $N_p=1001$. Top row: ground truth, bottom row: reconstruction. Global artifacts are not visible.}
\label{fig:frp}
\end{figure}

\begin{figure}[h]
{\centerline{
{\vbox{
{\hbox{
{\epsfig{file={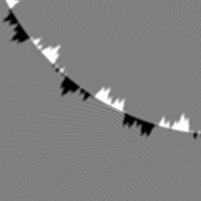}, width=2.5cm}}
{\epsfig{file={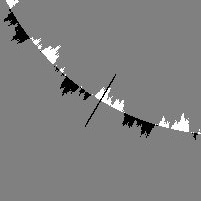}, width=2.5cm}}
{\epsfig{file={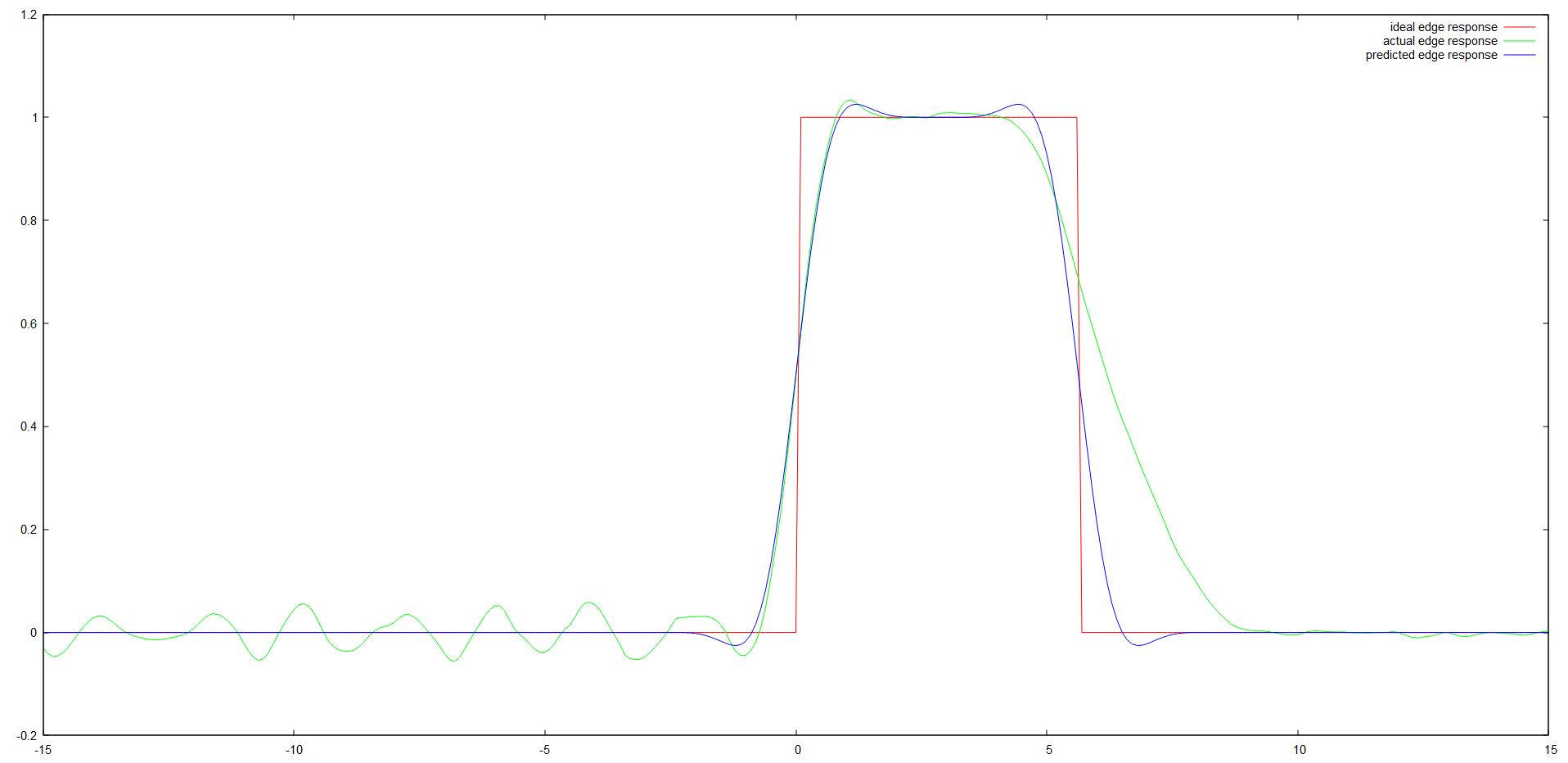}, width=6cm}}
}}
{\hbox{
{\epsfig{file={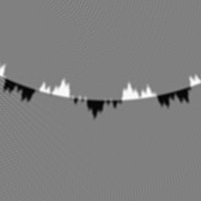}, width=2.5cm}}
{\epsfig{file={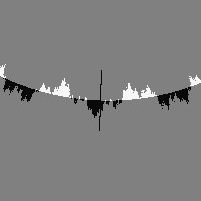}, width=2.5cm}}
{\epsfig{file={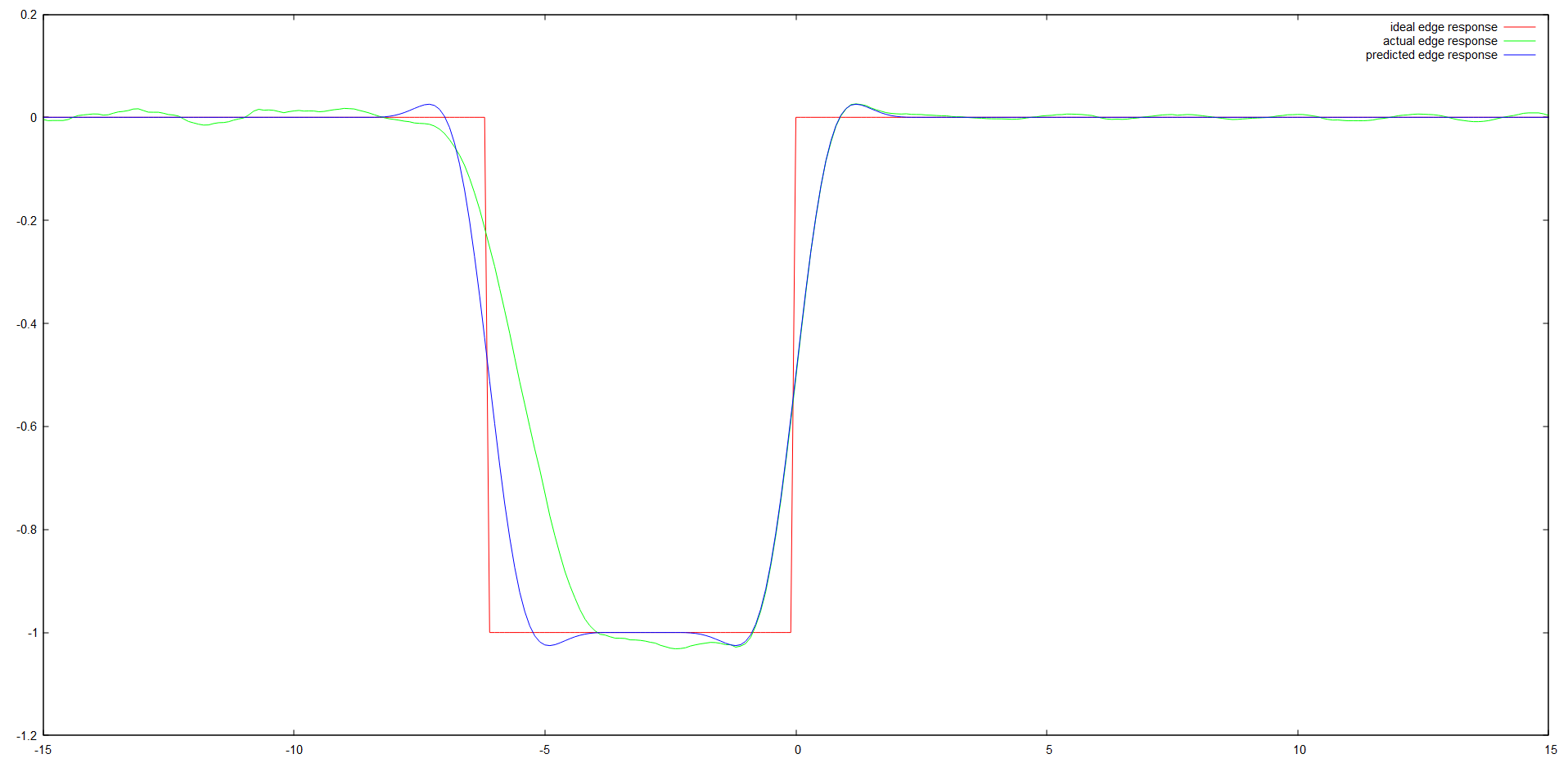}, width=6cm}}
}}
}}}}
\caption{ROI in the fractal phantom, $\ga=1/2$, $N_p=501$. Left column: reconstruction, middle column: ground truth with the location of the profile shown, right column: profiles along the line indicated in the middle panel. Top row: $\al=0.33\pi$, bottom row: $\al=0.49\pi$.}
\label{fig:ROI_frac_33_49_500}
\end{figure}

\begin{figure}[h]
{\centerline{
{\vbox{
{\hbox{
{\epsfig{file={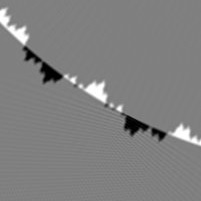}, width=2.5cm}}
{\epsfig{file={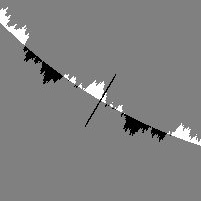}, width=2.5cm}}
{\epsfig{file={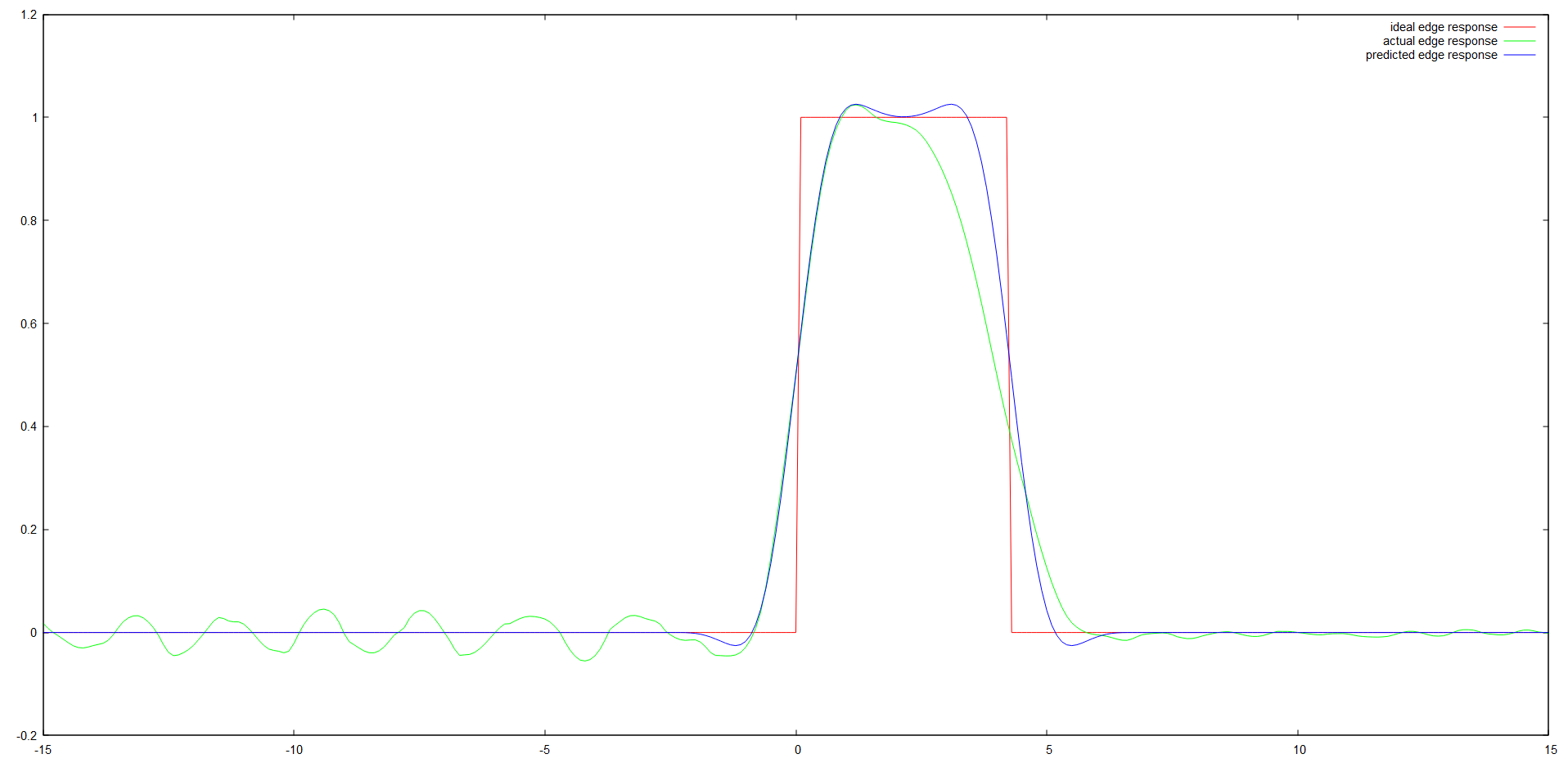}, width=6.0cm}}
}}
{\hbox{
{\epsfig{file={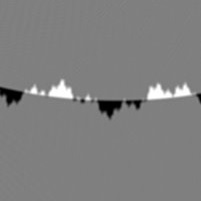}, width=2.5cm}}
{\epsfig{file={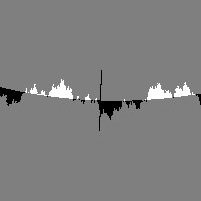}, width=2.5cm}}
{\epsfig{file={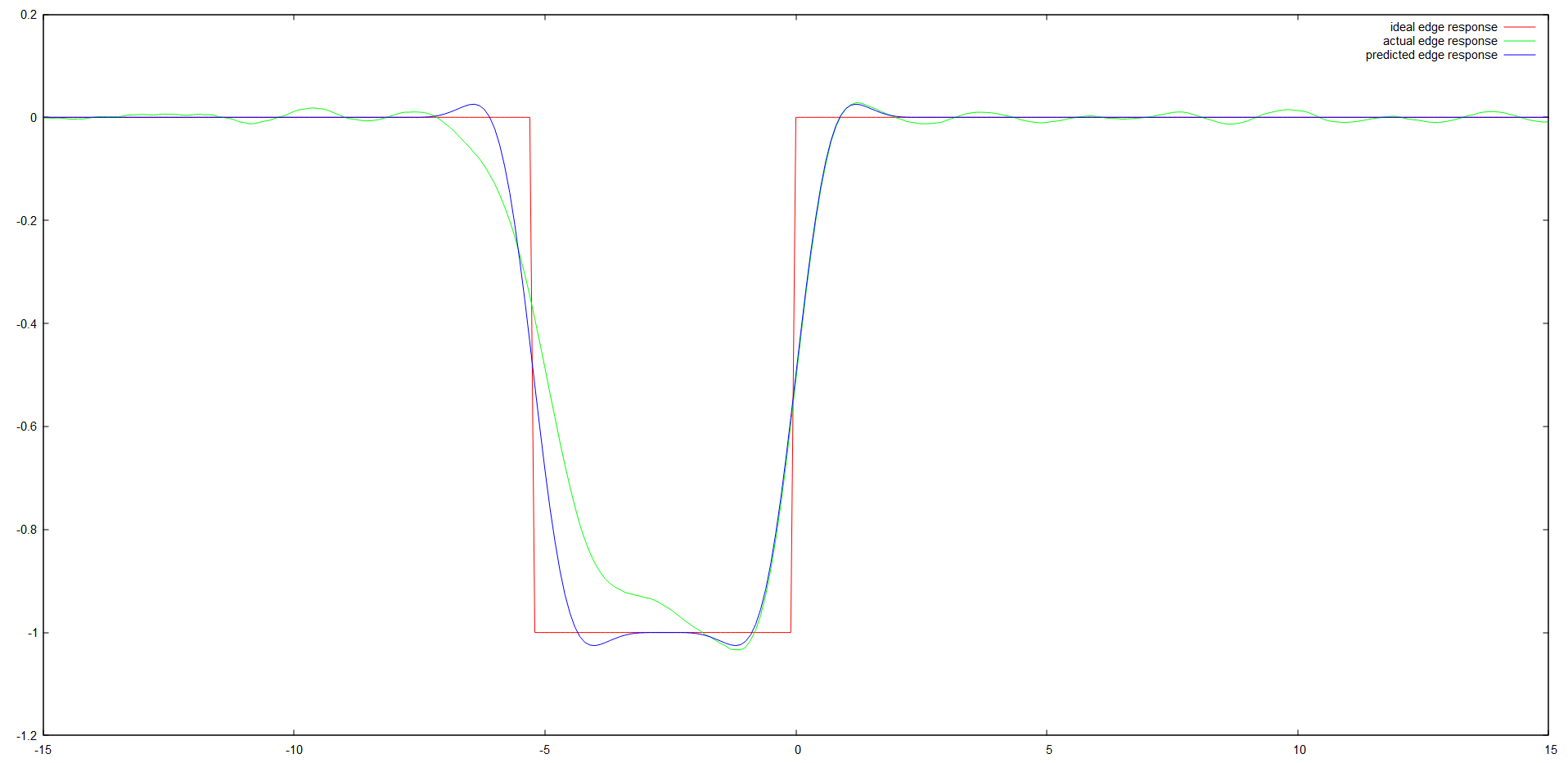}, width=6.0cm}}
}}
}}}}
\caption{ROI in the fractal phantom, $\ga=1/2$, $N_p=1001$. Left column: reconstruction, middle column: ground truth with the location of the profile shown, right column: profiles along the line indicated in the middle panel. Top row: $\al=0.33\pi$, bottom row: $\al=0.49\pi$.}
\label{fig:ROI_frac_33_49_1000}
\end{figure}

Comparing Figures~\ref{fig:ROI_op_500} with \ref{fig:ROI_frac_33_49_500} and \ref{fig:ROI_op_1000} with \ref{fig:ROI_frac_33_49_1000}, we see that the convergence of the reconstruction to the DTB in Theorem~\ref{main-res} is slower for smaller values of $\ga$. The convergence is the fastest for the globally smooth boundary (close to the zero coordinate in the plots on the right in Figures~\ref{fig:ROI_op_500}, \ref{fig:ROI_op_1000}, \ref{fig:ROI_frac_33_49_500}, and \ref{fig:ROI_frac_33_49_1000}), slower when $\ga=1$, and the slowest - when $\ga=0.5$.

\section{New DTB. Numerical experiments II}\label{sec numexp II}

From \eqref{data_eps} and \eqref{recon-orig} it is clear that we can write the reconstruction in the form
\be\label{via_kernel}
\frec(x)=-\frac{\Delta\al}{2\pi}\frac1{\e^2}\sum_{|\al_k|\le \pi/2}\sum_j \CH\ik'\left(\frac{\vec\al_k\cdot x-p_j}\e\right) \iint w\left(\frac{p_j-\vec\al_k\cdot y}{\e}\right)f_\e(y)\dd y.
\ee
Arguing formally, the sums with respect to $k$ and $j$ become integrals, and we get
\be\label{via_kernel_lim}\begin{split}
&\lim_{\e\to0}\left(\frec(x)-\frac{1}{\e^2}\iint K\left(\frac{x-y}\e\right)f_\e(y)\dd y\right)=0,\\
&K(z):=-\frac{1}{2\pi}\int_0^\pi (\CH\ik'*w)(\vec\al\cdot z) \dd\al.
\end{split}
\ee
Writing \eqref{via_kernel_lim} in the form of \eqref{DTB new use} we have
\be\label{new DTB}\begin{split}
&DTB_{new}(\check x,\e):=\frac{1}{\e^2}\iint K\left(\frac{(x_0+\e\check x)-y}\e\right)f_\e(y)\dd y,\\
&\frec(x_0+\e\check x)=DTB_{new}(\check x,\e)+\text{error term}.
\end{split}
\ee
Obviously, $K$ is radial and compactly supported. This follows, because $\hat K(\al,t)=(\ik*w)(t)$ is radial, compactly supported, and even (i.e., in the range of the Radon transform). Recall that $\hat K$ denotes the Radon transform of $K$. 

We see that \eqref{new DTB} is easy to implement and easy to analyze in order to investigate the resolution of reconstruction from discrete data. Also, equation \eqref{via_kernel_lim}, if it is correct, implies the result proven earlier. Intuitively, as $\e\to0$, the boundaries $\s$ and $\s_\e$ become locally flat, and we get Theorem~\ref{main-res}. To see this fact rigorously, start similarly to \eqref{g-fn}:
\be\label{new DTB deriv}\begin{split}
DTB_{new}(\check x,\e)=&\frac1{\e^2}\int_{-a}^a\int_{0}^{H_\e(\theta)} K\left(\check x-\frac{y_*(\theta)+t\vec\Theta}\e\right)F(\theta,t) \dd t \dd\theta \\
=&\int_{-a/\e}^{a/\e}\int_0^{\e^{-1}H_\e(\e\hat\theta)} K\left(\check x-\frac{y_*(\e\hat\theta)+\e\hat t\vec\Theta}\e\right)F(\e\hat\theta,\e\hat t) \dd\hat t \dd\hat\theta,
\end{split}
\ee
where $F$ is the same as in \eqref{g-fn}, and we used that $x_0=y_*(0)=0$. Since $K$ is compactly supported, it is clear that $\hat\theta$ is confined to a bounded set, and
\be\label{new DTB deriv st2}\begin{split}
DTB_{new}(\check x,\e)
=&F(0,0)\int_{\br}\int_0^{H_0(\e^{1/2}\hat\theta)} K\left(\check x-\left(\hat\theta y_*'(0)+\hat t\vec\Theta_0\right)\right) \dd\hat t \dd\hat\theta+O(\e),
\end{split}
\ee
where $H_0(\e^{1/2}\hat\theta)=H_0(0)+O(\e^{\ga/2})$. Using that $F(0,0)=\Delta f(x_0)R_0$ and $|y_*'(0)|=R_0$, we compute
\be\label{new DTB deriv st3}\begin{split}
DTB_{new}(\check x,\e)
=&\Delta f(x_0)\int_{\br}\int_0^{H_0(0)} \hat K\left(\vec\Theta_0,\vec\Theta_0\cdot\check x-\hat t\right) \dd\hat t \dd\hat\theta+O(\e^{\ga/2}).
\end{split}
\ee
Recalling that $\hat K=\ik*w$, the assertion follows.

Moreover, the local convergence of $\s_\e$ to a flat line segment (which is described in terms of the convergence $H_0(\e^{1/2}\hat\theta)\to H_0(0)$) is slower the lower the value of $\ga$ is, which matches the observations in Section~\ref{sec numexp I}. Thus, \eqref{new DTB} has the potential to be a more accurate result than that given in Theorem~\ref{main-res}. Our numerical experiments confirm this conjecture. We implemented the kernel $K(z)$ and convolved it with $f_\e$ to compute $\DTB_{new}$ for the same two values of $\e$. The graphs along the same lines as in Figures~\ref{fig:ROI_frac_33_49_500}, \ref{fig:ROI_frac_33_49_1000}, are shown in Figures~\ref{fig:ROI_with2D_33_49_500}, \ref{fig:ROI_with2D_33_49_1000}, respectively. By comparing the profiles we see that the latter are significantly more accurate than the former. 

\begin{figure}[h]
{\centerline{
{\vbox{
{\epsfig{file={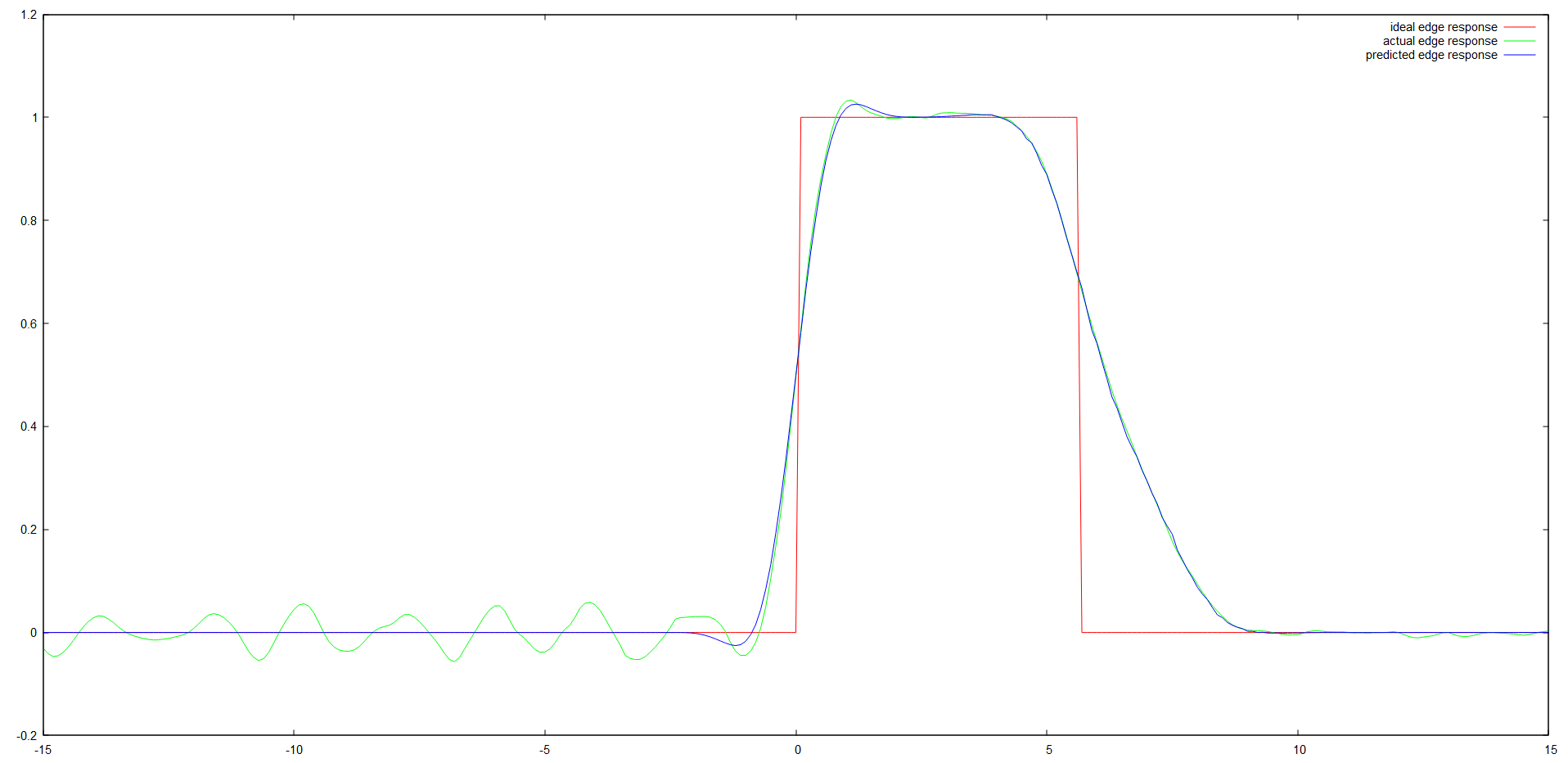}, width=6.0cm}}
{\epsfig{file={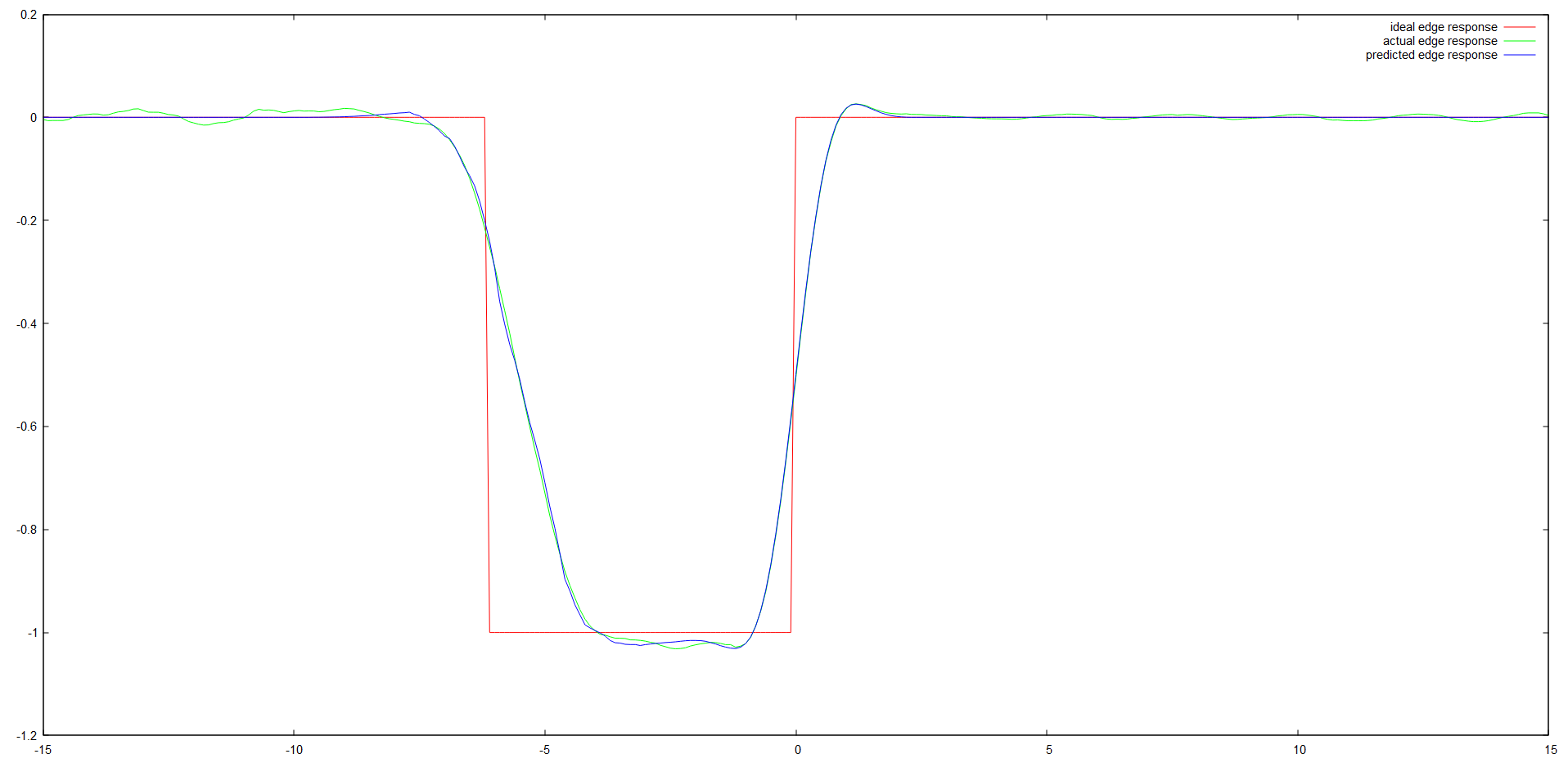}, width=6.0cm}}
}}}}
\caption{ROI in the fractal phantom of Figure~\ref{fig:ROI_frac_33_49_500}, $\ga=0.5$, $N_p=501$. Left: $\al=0.33\pi$, right: $\al=0.49\pi$.}
\label{fig:ROI_with2D_33_49_500}
\end{figure}

\begin{figure}[h]
{\centerline{
{\vbox{
{\epsfig{file={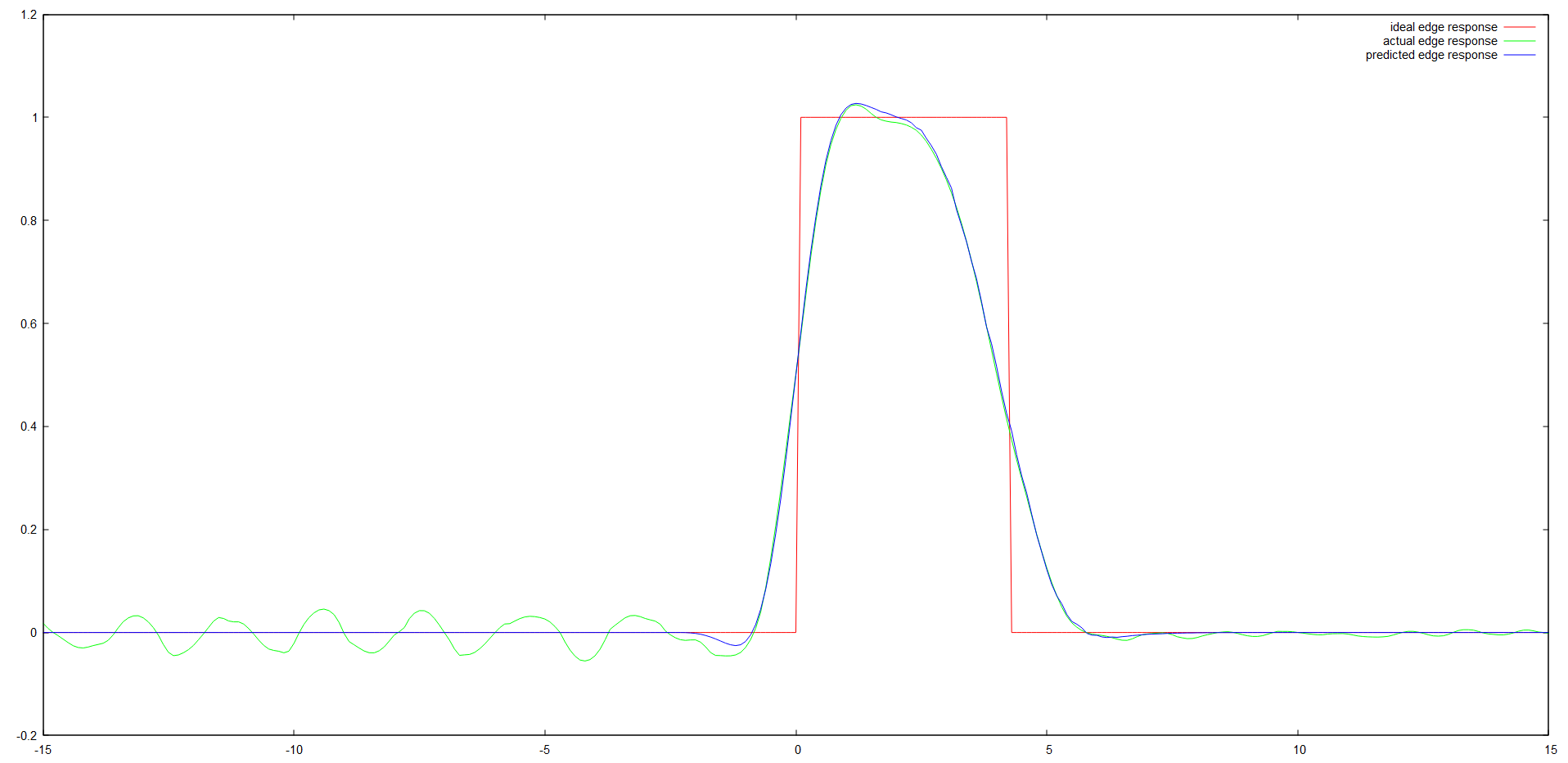}, width=6cm}}
{\epsfig{file={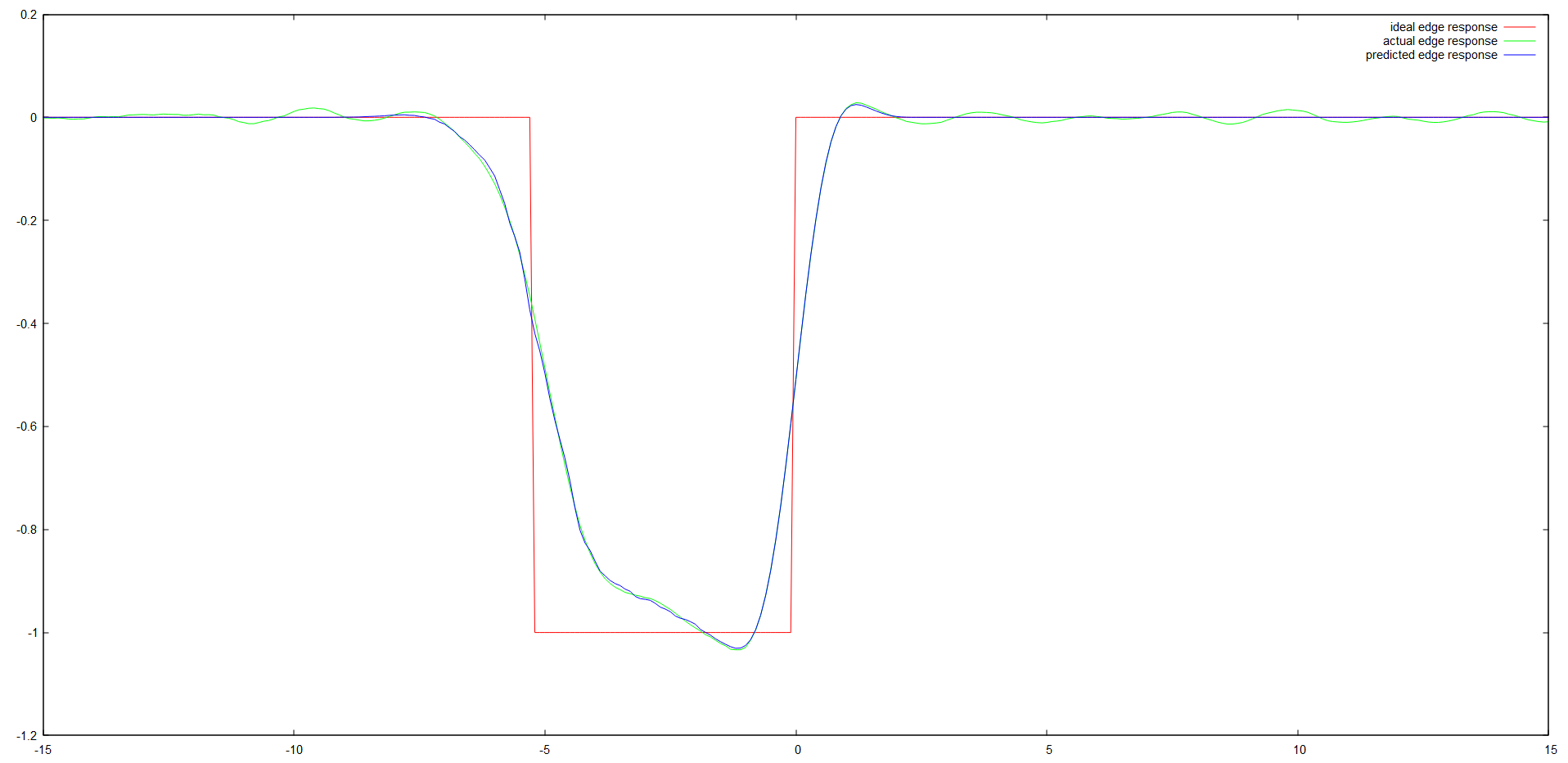}, width=6cm}}
}}}}
\caption{ROI in the fractal phantom of Figure~\ref{fig:ROI_frac_33_49_1000}, $\ga=0.5$, $N_p=1001$. Left: $\al=0.33\pi$, right: $\al=0.49\pi$.}
\label{fig:ROI_with2D_33_49_1000}
\end{figure}

Due to the division by $\e$, the sums with respect to $j$ and $k$ in \eqref{via_kernel} involve stepsizes that do not go to zero as $\e\to0$. Therefore, replacing the sums with integrals in \eqref{via_kernel} appears counterintuitive. 
In general, \eqref{via_kernel_lim} is indeed false. This can be seen, for example, as follows. Since the kernel $K$ is compactly supported, formula \eqref{via_kernel_lim} does not explain non-local artifacts, which are known to arise in case $f$ has a singularity across a line segment. Thus, at best, \eqref{via_kernel_lim} holds only for certain classes of functions. Our numerical experiments suggest that the formula holds for functions described in Section~\ref{sec:a-prelims}. The complete proof of \eqref{via_kernel_lim} is complicated and beyond the scope of this paper. In this section we state one result along these lines and formulate a conjecture about the rate of convergence in \eqref{via_kernel_lim}. A complete explanation of our numerical results requires establishing the convergence rates in \eqref{final-lim-v2} and \eqref{via_kernel_lim}, and making sure that the latter is faster than the former. 

Here we strengthen IK1 by requesting that $\tilde\ik(\la)=O(|\la|^{-3})$, $\la\to\infty$. The Keys interpolation kernel in \eqref{keys} satisfies this assumption. Consider the function
\be\label{recon-ker-0}
\psi(q,t):=\sum_j (\CH\ik')(q-j)w(j-q-t).
\ee
Then
\be\label{psi-props-0}\begin{split}
&\psi(q,t)=\psi(q+1,t),\  q,t\in\br;\quad
\psi(q,t)=O(t^{-2}),\ t\to\infty,\ q\in\br,\\
&\int \psi(q,t)\dd t\equiv 0,\ q\in\br.
\end{split}
\ee
The last property follows from IK2 (see \eqref{exactness}).
By \eqref{psi-props-0}, we can represent $\psi$ in terms of its Fourier series:
\be\label{four-ser}\begin{split}
\psi(q,t)&=\sum_m \tilde\psi_m(t) e(-mq),\ e(q):=\exp(2\pi i q),\\
\tilde\psi_m(t) &=\int_0^1 \psi(q,t)e(mq)\dd q=\int_\br (\CH\ik')(q) w(-q-t)e(mq)\dd q.
\end{split}
\ee
Due to the assumptions that $\ik,w\in C_0^2(\br)$ and $\tilde\ik(\la)=O(|\la|^{-3})$ we have
\be\label{four-coef-bnd}
|\tilde\psi_m(t)| \le c(1+m^2)^{-1}(1+t^2)^{-1},
\ee
for some $c$, so the Fourier series for $\psi$ converges absolutely. Indeed, if $t$ is restricted to any compact set, the result follows because $\widetilde{(\CH\ik')}(\la),\tilde w(\la)=O(\la^{-2})$ implies 
\be\label{convol decay}\begin{split}
\left|\int |\mu|\tilde\ik(\mu)\tilde w(\mu-\la)e^{i(\mu-\la)t}\dd\mu\right|
\le \int |\mu\tilde\ik(\mu)\tilde w(\mu-\la)|\dd\mu=O(\la^{-2}), \la=2\pi m\to\infty.
\end{split}
\ee
If $|t|\ge c$ for some $c\gg1$ sufficiently large, integrate by parts twice and use that 
\be\label{deriv bnd}
\max_{q} |(\pa/\pa q)^2((\CH\ik')(q)w(-q-t))|=O(t^{-2}),\ t\to\infty.
\ee
Differentiation by parts works, because $(\CH\ik')(q)$ is smooth in a neighborhood of any $q$ such that $w(-q-t)\not=0$.

From \eqref{via_kernel}, \eqref{recon-ker-0},  and \eqref{four-ser}, the reconstructed image becomes
\be\label{recon-ker-v2}
\begin{split}
&\frec(x)
=-\frac{\Delta\al}{2\pi}\sum_m \sum_k e\left(-m q_k\right)A_m(\al_k,\e),\ \sum_k:=\sum_{|\al_k|\le\pi/2}, \\
&q_k:=\frac{\vec\al_k\cdot x}\e,\ A_m(\al,\e):=\e^{-2}\iint \tilde\psi_m\left(\frac{\vec\al\cdot (y-x)}\e\right)f_\e(y)\dd y.
\end{split}
\ee
Suppose first that $x_0=y_*(0)$.
To obtain \eqref{via_kernel_lim}, we should be able to replace the sum with respect to $k$ by an integral with respect to $\al$ and ignore all $m\not=0$ terms. The results in Sections~\ref{sec:beg proof} and \ref{sec: rem sing} suggest that the rate in \eqref{final-lim-v2} is $O(\e^{\ga/2})$, cf. \eqref{partI-three}. Thus, to establish that the new DTB is more accurate, we need to show that 
\be\label{extra terms}
\begin{split}
\Delta\al\sum_{m\not=0}& \sum_k e\left(-m q_k\right)A_m(\al_k,\e)=o(\e^{\ga/2}),\  \\
\frac{\Delta\al}{2\pi\e^2} \sum_k&
\iint \int_{\al_k-\Delta\al/2}^{\al_k+\Delta\al/2}  \left(\tilde\psi_0\left(\frac{\vec\al\cdot (y-x)}\e\right)-\tilde\psi_0\left(\frac{\vec\al_k\cdot (y-x)}\e\right)\right) \dd\al \\
&\times f_\e(y)\dd y=o(\e^{\ga/2}).
\end{split}
\ee

Suppose now $x_0\not=y_*(0)$, which is equivalent to assuming $x_0\not\in\s$. In this case $f$ is smooth near $x_0$, so we should expect that $DTB(\check x,\e)\equiv0$ for all $\e>0$ sufficiently small. Since $K(z)$ is compactly supported, $\iint K((x-y)/\e)f_\e(y)\dd y\equiv0$ for all $\e>0$ sufficiently small and all $x$ sufficiently close to $x_0$, precisely as expected. Hence, there is no need to single out the term $m=0$ (which previously gave the only nonzero contribution) in \eqref{recon-ker-v2}, because the entire sum should go to zero sufficiently fast as $\e\to0$. The following lemma states that this is indeed the case. Its proof is in Appendix~\ref{sec new ker prf}.

\begin{lemma}\label{lem:partial res} Pick $x_0\not\in\s$ such that no line through $x_0$, which intersects $\s$, is tangent to $\s$. This includes the endpoints of $\s$, in which case the one-sided tangents to $\s$ are considered. Suppose the level sets of $H_0$ are well-behaved, i.e. there exist $\rho,L_0>0$ independent of $\hat t$ such that any open interval of any length $L\ge L_0$ contains no more than $\rho L$ points from $H_0^{-1}(\hat t)$ for any $\hat t$. Under the assumptions of Theorem~\ref{main-res}, one has
\be\label{extra terms top}
\frec(x)=O(\e^{1/2}\ln(1/\e)),\ \e\to0,
\ee
uniformly with respect to $x$ in a sufficiently small neighborhood of $x_0$.
\end{lemma}

Using numerical evidence and the above lemma as a guide, we state the following conjecture. 
\begin{conjecture} Pick any generic $x_0$. Suppose the level sets of $H_0$ are well-behaved, as defined in Lemma~\ref{lem:partial res}. Under the assumptions of Theorem~\ref{main-res}, one has
\be\label{full conj}
\frec(x_0+\e\check x)=\frac1{\e^2}\iint K\left(\frac{(x_0+\e\check x)-y}\e\right)f_\e(y)\dd y+O(\e^{1/2}\ln(1/\e)),\ \e\to0,
\ee 
where the big-$O$ term is uniform with respect to $\check x$ in any compact set.
\end{conjecture}

To prove the conjecture one has to consider the case $x_0\in\s$ as well as the case $x_0\not\in\s$, when a line through $x_0$ is tangent to $\s$. The assumptions in Lemma~\ref{lem:partial res} and the conjecture are not vacuous in the following sense. If all the level sets of a function are well-behaved, this does not imply that the function is Lipschitz continuous. In Section~\ref{bad fn} we present an example of a function on $[0,\infty)$ (the Schwarz function \cite{Thim2003}), which is strictly monotonically increasing, locally Holder continuous with exponent $\ga\in(0,1)$, and is not locally Holder continuous with any exponent $\ga'>\ga$ on any interval. Using such a function as a building block, one can create a function with well-behaved level sets in the sense of Lemma~\ref{lem:partial res}. Due to the limited smoothness of such $H_0$, the derivation in Sections~\ref{sec:beg proof}, \ref{sec: rem sing} for the original DTB cannot guarantee convergence faster than $O(\e^{\ga/2})$, which is slower than the conjectured rate.

Our numerical experiments show that even when the level sets of a function are not well-behaved (see \cite{Rezakhanlou1988, Yu2020a} regarding level sets of the Weirstrass function), $\DTB_{new}$ still appears to exhibit rapid convergence faster than $O(\e^{\ga/2})$.



\appendix

\section{Proof of Lemma~\ref{lem:g1-props}}\label{sec:lemI}

We prove the lemma in the more complicated case $g_*=g$. The case $g_*=g_l$ is proven along the same lines, but many of the steps are simpler. Pick any $\al\in (-a,a)$. By construction, 
\be\label{aux-est}
\vec\al\cdot (y_*(\al+\auxang)-y_*(\al))=
\auxang^2 \psi(\auxang),\ \psi\in C^2([-a,a]),
\ee
where, by \eqref{short-cond},
\be\label{psi-props}\begin{split}
&\psi(\auxang)\ge c_1\text{ and } |\psi'(\auxang)|,|\psi''(\auxang)|\le c_2 \text{ if } |\al|,|\al+\auxang|\le a;\\ &\psi(\auxang)=(R(\al)/2)+O(|\auxang|),\auxang\to0,
\end{split}
\ee
for some $c_{1,2}>0$. The dependence of $\psi$ on $\al$ is irrelevant and omitted from notation. Therefore, by \eqref{g-fn-v2} 
\be\label{left-supp-g2}
\gc(\al,\hat p)\equiv 0 \text{ for }\hat p<c
\ee 
for some $c<0$ independent of $\al\in(-a,a)$. 

Suppose now $\hat p\to+\infty$. Since $w$ is compactly supported and $H_0$ is bounded, we can find $c>0$ sufficiently large so that the domain of integration with respect to $\auxang$ in \eqref{g-fn-v2} is contained inside the union of two non-intersecting intervals, whose endpoints are computed by solving 
\be\label{endpts}
|\tth|\psi^{1/2}(\e^{1/2}\tth)=(\hat p\pm c)^{1/2}.
\ee 
The positive pair of solutions $\tth_\pm^+>0$ determines one interval, and the negative pair $\tth_\pm^-<0$ - the other. Consider, for example, the positive pair. By \eqref{psi-props},
we get that $\dd\left[ \tth\psi^{1/2}(\e^{1/2}\tth)\right]/\dd\tth\ge c$ for some $c>0$ as long as $|\al|,|\al+\auxang|\le a$ and $\e>0$ is sufficiently small. Hence $\tth_+^+-\tth_-^+=O((\hat p+c)^{1/2}-(\hat p-c)^{1/2})=O(\hat p^{-1/2})$. The same argument applies to the negative pair, and \eqref{g-as-st2-pm} follows.

To prove the last assertion of the lemma, set $\gc=\gc^++\gc^-$, where
\be\label{g2-asympt}\begin{split}
\gc^+(\al,\hat p)=\int_0^{\e^{-1/2}(a-\al)}
\int^{H_0(\tilde\al+\tth)}_0
& w\left(\hat p-\tth^2 \psi(\e^{1/2}\tth)-\hat t\cos(\e^{1/2}\tth)\,\right)\\
&\times F(\al+\e^{1/2}\tth,\e\hat t) \dd\hat t \dd\tth,
\end{split}
\ee
$\tilde\al=\e^{-1/2}\al$, and $\gc^-$ is defined similarly by integrating over $(\e^{-1/2}(-a-\al),0]$ with respect to $\tth$. 
First, we consider $\gc^+$, so $\tth>0$.
Introduce the variable $\hat s=\tth^2\psi(\e^{1/2}\tth)$. 
By \eqref{psi-props}, $\tth'(s)$ is uniformly bounded for all $\e>0$ small enough whenever $\hat s$ is bounded away from zero. To indicate the dependence of $\tth(\hat s)$ on $\e$ we write $\tth_\e(\hat s)$.
Change variables $\tth=\tth_\e(\hat s)$ in \eqref{g2-asympt}:
\be\label{g2-st2}\begin{split}
\gc^+(\al,\hat p)=&\int_{\br}
\int^{H_0(\tilde\al+\tth)}_0
 w\left(\hat p-\hat s-\hat t\cos(\e^{1/2}\tth)\,\right)F(\al+\e^{1/2}\tth,\e\hat t) \dd\hat t\, \tth_\e'(\hat s)\dd\hat s,\\
\tth=&\tth_\e(\hat s^{1/2}).
\end{split}
\ee
Here we extended the integration with respect to $\hat s$ to $\br$. Even though $\tth_\e'(\hat s)$ is not defined for $\hat s>0$ sufficiently large, this is irrelevant because $F(\cdot)\equiv 0$ for such $\hat s$. Changing the lower limit does not change the integral either, because $w$ is compactly suported and $\hat p\to+\infty$. Using the argument following \eqref{g-fn-sigma3}, such an extension does not affect the smoothness of $F$.  
Similarly, 
\be\label{g2-dp}\begin{split}
\gc^+(\al,\hat p+\Delta\hat p)=&\int_\br
\int^{H_0(\tilde\al+\tth)}_0
w\left(\hat p-\hat s-\hat t\cos(\e^{1/2}\tth)\,\right)\\
&\times F(\al+\e^{1/2}\tth,\e\hat t) \dd\hat t\, \tth_\e'(\hat s+\Delta \hat p)\dd\hat s,\ \tth=\tth_\e(\hat s+\Delta \hat p).
\end{split}
\ee
As before, from \eqref{psi-props} we obtain
\be\label{del-th}\begin{split}
\tth_\e^{(k)}(\hat p)=O\left(\hat p^{1/2-k}\right),\ \hat p\to+\infty,\ k=0,1,2,
\end{split}
\ee
uniformly in $\e$. 

By \eqref{del-th}, dropping $\Delta\hat p$ in the argument of $\tth_\e$ in the argument of $H_0$ in the upper limit of the inner integral in \eqref{g2-dp} leads to an error of magnitude 
\be\label{main-err}
\hat p^{-1/2}O\left((|\Delta\hat p|/\hat p^{1/2})^{\ga}\right). 
\ee
Recall that $\Delta\hat p=O(\hat p^\de)$, $\de<1/2$, $\hat p\to+\infty$. Dropping $\Delta\hat p$ in $\tth_\e$, which is located in the arguments of $w$ and $F$, leads to an error of magnitude $(\e/\hat p)^{1/2}O\left(|\Delta\hat p|/\hat p^{1/2}\right)$. Dropping $\Delta\hat p$ from $\tth_\e'$ leads to an error of magnitude $O\left(|\Delta\hat p|/\hat p^{3/2}\right)$.

Under our assumptions $\ga<1$ and $\de<1/2$, so all the error terms are dominated by \eqref{main-err}. 
Subtracting \eqref{g2-st2} from \eqref{g2-dp} we prove that $\gc^+(\al,\hat p)$ satisfies the estimate in
\eqref{delg-bnd-st2-pm}.
Similar arguments and similar estimates hold for $\gc^-(\al,\hat p)$ as well, and \eqref{delg-bnd-st2-pm} is proven. 

\section{Proof of Lemma~\ref{lem:Psi-asympt}}\label{sec:prf-lemPsi}

Using that $(\CH\ik')(t)=O(t^{-2})$, $t\to\infty$, we obtain for some $c$ using \eqref{left-supp-pm}, \eqref{g-as-st2-pm}:
\be\label{G-left}
|\Psil(\tilde\al,\hat p,q)|\le c\sum_{j\ge -c} \frac1{1+(|\hat p|+j)^2}\frac1{1 +|j|^{1/2}}
=O(|\hat p|^{-3/2}),\ \hat p\to-\infty,\,q\in[0,1).
\ee
Fix any $\de$, $0<\de<1/2$. Similarly to \eqref{G-left}, we can show that
\be\label{G-right}
\sum_{|j-\hat p|\ge p^\de} (\CH\ik')\left(\hat p-j\right)\gl(\al,j-q)
=O(\hat p^{-(1/2+\de)}),\ \hat p\to +\infty,\,q\in[0,1).
\ee
Split the remaining sum into two:
\be\label{G-interior}\begin{split}
\sum_{|j-\hat p|< \hat p^\de}& (\CH\ik')\left(\hat p-j\right)\gl(\al,j-q)\\
=&\sum_{|j-\hat p|< \hat p^\de} (\CH\ik')\left(\hat p-j\right)(\gl(\al,j-q)-\gl(\al,\hat p-q))\\
&+\gl(\al,\hat p-q)\sum_{|j-\hat p|< \hat p^\de} (\CH\ik')\left(\hat p-j\right)=:S_1+S_2.
\end{split}
\ee
Clearly, $\CH\ik'(t)=O(t^{-2})$, $t\to\infty$.
Combining with \eqref{holder}, \eqref{g-as-st2-pm}, and \eqref{delg-bnd-st2-pm}, we find
\be\label{S1est}\begin{split}
S_1=& O\left(\hat p^{\ga(\de-1/2)-1/2}\right)\sum_{|j-\hat p|< \hat p^\de} (\CH\ik')(\hat p-j)
=O\left(\hat p^{\ga(\de-1/2)-1/2}\right),\ \hat p\to+\infty.
\end{split}
\ee
Moreover, by the exactness of $\ik$ (property IK2),
\be\label{ker-pr}
\sum_{|j-\hat p|< \hat p^\de} \ik'\left(r-j\right)\not=0\text{ only if }|r-(\hat p-\hat p^\delta)|\le c \text{ or } |r-(\hat p+\hat p^\delta)|\le c
\ee
for some $c$. Hence,
\be\label{S2est}\begin{split}
S_2=O\left(\hat p^{-1/2}(1+\hat p^{2\de})^{-1}\right)=O\left(\hat p^{-1/2-2\de}\right),\ 
\hat p\to +\infty,\,q\in[0,1).
\end{split}
\ee
Combining \eqref{G-right}, \eqref{G-interior}, \eqref{S1est}, and \eqref{S2est} gives
\be\label{Psi-right}
|\Psil(\tilde\al,\hat p,q)|=O(\hat p^{-(1/2+\de)}),\ \de=(\ga/2)/(\ga+1),\ \hat p\to+\infty,\,q\in[0,1).
\ee
The choice of $\de$ in \eqref{Psi-right} satisfies $0<\de<1/2$ and provides the fastest guaranteed rate of decay of $\Psil$. Combining \eqref{G-left} and \eqref{Psi-right} (and replacing $\de$ with $\de/2$ for notational convenience) proves the lemma. 

\section{Proof of Lemma~\ref{lem:partial res}}\label{sec new ker prf}

Pick any $x$ sufficiently close to $x_0$. All the estimates below are uniform with respect to $x$, so the $x$-dependence of various quantities is frequently omitted from notation. 

Let $\Omega_x$ be the set of all $\al\in[-\pi/2,\pi/2]$ such that the lines $\{y\in\br^2:\,(y-x)\cdot\vec\al=0\}$ intersect $\s$. Let $\theta(\al)$, $\al\in\Omega_x$, be determined by solving $(y_*(\theta)-x)\cdot\vec\al=0$. By assumption, the intersection is transverse for any $\al\in\Omega_x$ (up to the endpoints). Hence 
$|\theta'(\al)|=|y_*(\theta(\al))-x|/|\vec\al\cdot y_*'(\theta(\al))|$ and 
\be\label{theta deriv min}
0<\min_{\al\in\Omega_x}|\vec\al\cdot y_*'(\theta(\al))|,\ 
0<\min_{\al\in\Omega_x}|\theta'(\al)|\le \max_{\al\in\Omega_x}|\theta'(\al)|<\infty.
\ee

Transform the expression for $A_m$ (cf. \eqref{recon-ker-v2}) similarly to \eqref{g-fn}:
\be\label{A-simpl}
\begin{split}
A_m(\al,\e)=\frac1\e\int_{-a}^a\int_0^{\e^{-1}H_\e(\theta)}\tilde\psi_m\left(\frac{\vec\al\cdot (y_*(\theta)-x)}\e+\hat t\cos(\theta-\al)\right)F(\theta,\e \hat t)\dd\hat t\dd\theta,
\end{split}
\ee
where $F$ is the same as in \eqref{g-fn}. Setting $\tilde\theta=(\theta-\theta(\al))/\e^{1/2}$,
\eqref{A-simpl} becomes
\be\label{A-simpl_2}
\begin{split}
A_m(\al,\e)=&\e^{-1/2}\int\int_0^{\e^{-1}H_\e(\theta)}\tilde\psi_m\left(\frac{\vec\al\cdot (y_*(\theta)-y_*(\theta(\al)))}\e+\hat t\cos(\theta-\al)\right)\\
&\times F(\theta,\e \hat t)\dd\hat t\dd\tilde \theta,\quad
\theta=\theta(\al)+\e^{1/2}\tilde\theta.
\end{split}
\ee
Due to \eqref{four-coef-bnd}, we can integrate with respect to $\tilde\theta$ over any fixed neighborhood of $0$:
\be\label{A-simpl-st2}
\begin{split}
&A_m(\al,\e)\\
&=\e^{-1/2}\int_{-\de}^{\de}\int_0^{\e^{-1}H_\e(\theta)}\tilde\psi_m\left(\frac{\vec\al\cdot y_*'(\theta(\al))}{\e^{1/2}}\tilde\theta+O(\tilde\theta^2)+\hat t\cos(\theta(\al)-\al)+O(\e^{1/2})\right)\\
&\quad\times \left(F(\theta(\al),0)+O(\e^{1/2})\right)\dd\hat t\dd\tilde \theta+(1+m^2)^{-1}O(\e^{1/2})\\
&=\frac{F(\theta(\al),0)}{\e^{1/2}}\int_{-\de}^{\de}\int_0^{\e^{-1}H_\e(\theta)}\tilde\psi_m\left(\frac{\vec\al\cdot y_*'(\theta(\al))}{\e^{1/2}}\tilde\theta+O(\tilde\theta^2)+\hat t\cos(\theta(\al)-\al)\right)\dd\hat t\dd\tilde \theta\\
&\quad+(1+m^2)^{-1}O(\e^{1/2}),\qquad \theta=\theta(\al)+\e^{1/2}\tilde\theta,
\end{split}
\ee
for some $\de>0$ sufficiently small. Here we have used that 
\be\label{psi_psipr}
\tilde\psi_m(t),\tilde\psi_m'(t)=(1+m^2)^{-1}O(t^{-2}), \ t\to\infty,
\ee
which follows from \eqref{recon-ker-0}, \eqref{four-ser}. Similarly, it is easy to see that the term $O(\tilde\theta^2)$ can be omitted from the argument of $\tilde\psi_m$ without changing the error term, and we find
\be\label{A-simpl-st3}
\begin{split}
A_m(\al,\e)
=&\frac{F(\theta(\al),0)}{\e^{1/2}}\int_{-\de}^{\de}\int_0^{H_0(\e^{-1/2}\theta(\al)+\tilde\theta)}\tilde\psi_m\left(a(\al)\e^{-1/2}\tilde\theta+b(\al)\hat t\right)\dd\hat t\dd\tilde \theta\\
&+(1+m^2)^{-1}O(\e^{1/2}),\quad 
a(\al):=\vec\al\cdot y_*'(\theta(\al)),\ b(\al):=\cos(\theta(\al)-\al).
\end{split}
\ee
By \eqref{theta deriv min}, $a(\al)$ is bounded away from zero on $\Omega_x$.  By the last equation in \eqref{psi-props-0}, $\int\tilde\psi_m(\hat t)\dd\hat t=0$ for all $m$, so we can replace the lower limit in \eqref{A-simpl-st3} with any value independent of $\tilde\theta$. Again, we use here that the contribution to the integral with respect to $\tilde\theta$ of the domain outside $(-\de,\de)$ is of the same magnitude as the error term in \eqref{A-simpl-st3}. We choose the lower limit to be $H_0(\e^{-1/2}\theta(\al))$. 
Define
\be\label{sum-k}
\begin{split}
B_m(\tilde\theta,\e)
:=&\Delta\al\sum_{\al_k\in\Omega_x} e\left(-m q_k\right) F(\theta(\al_k),0)\\
&\times\int_{H_0(s_k)}^{H_0(s_k+\tilde\theta)}\tilde\psi_m\left(a(\al_k)\e^{-1/2}\tilde\theta+b(\al_k)\hat t\right)\dd\hat t,\ s_k:=\e^{-1/2}\theta(\al_k).
\end{split}
\ee
The factor $\e^{-1/2}$ in front of the double integral in \eqref{A-simpl-st3} will be accounted for when integrating with respect to $\tilde\theta$ below (see \eqref{last bnd}).

Define similarly to \eqref{H-def}:
\be\label{H-def alt}\begin{split}
\chi_{t_1,t_2}(t):=&\begin{cases}1,&t_1\le t\le t_2,\\
0,&t\not\in [t_1,t_2],\end{cases} \text{ if $t_1<t_2$},\ \chi_{t_1,t_2}(t):=\begin{cases}-1,&t_2\le t\le t_1,\\
0,&t\not\in [t_2,t_1],\end{cases} \text{ if $t_2<t_1$},
\end{split}
\ee
and rewrite \eqref{sum-k} in the form
\be\label{sum-k alt}
\begin{split}
B_m(\tilde\theta,\e)
=&\int \Delta\al\sum_{\al_k\in\Omega_x}G_m(\e^{-1/2}\tilde\theta,\hat t,\al_k) e\left(-m q_k\right) \chi_{H_0(s_k),H_0(s_k+\tilde\theta)}(\hat t)\dd\hat t\\
=&\int W_{\e,m}(\tilde\theta,\hat t) \dd\hat t,
\end{split}
\ee
where
\be\label{G and a_k b_k}
\begin{split}
&G_m(r,\hat t,\al):=F(\theta(\al),0)\tilde\psi_m\left(a(\al)r+b(\al)\hat t\right),\\
&W_{\e,m}(\tilde\theta,\hat t):=
\Delta\al\biggl[\sum_{\substack{\al_k\in\Omega_x:\\H_0(s_k)< \hat t<H_0(s_k+\tilde\theta)}}
-\sum_{\substack{\al_k\in\Omega_x:\\H_0(s_k+\tilde\theta) < \hat t<H_0(s_k)}}\biggr]
G_m(\e^{-1/2}\tilde\theta,\hat t,\al_k) e\left(-m q_k\right).
\end{split}
\ee
By ignoring finitely many values of $\hat t$, we assume here and below that $\hat t\not= H_0(s_k)$ and $t\not=H_0(s_k+\tilde\theta)$ for any $k$.

Suppose, for example, that $\tilde\theta>0$. We say that $s_k$ and $s_j$ are {\it equivalent} for a given $\tilde\theta$, $\hat t$ and $\e$ if 
\be\label{equiv rel}\begin{split}
&\text{(a) the sets }(s_k,s_k+\tilde\theta)\cap H_0^{-1}(\hat t) \text{ and }(s_j,s_j+\tilde\theta)\cap H_0^{-1}(\hat t) \text{ are the same,} \\
&\text{(b) the number of points in this set is odd.}
\end{split}
\ee
Clearly, this is an equivalence relation, so it splits a subset of $s_k$ into equivalence classes $\Xi_n(\hat t,\e)$, $n=1,2,\dots, N(\hat t,\e)$. The requirement in (b) is necessary, because otherwise either $\hat t>H_0(s_k),H_0(s_k+\tilde\theta)$ or $\hat t<H_0(s_k),H_0(s_k+\tilde\theta)$.

The number of classes does not exceed the number of points in $(\min s_k,\max s_k+\de)\cap H_0^{-1}(\hat t)$, where the minimum and maximum are computed over $k$ such that $\al_k\in\Omega_x$. Since $\max s_k-\min s_k=O(\e^{-1/2})$, the assumption of the lemma implies that the number of classes satisfies $N(\hat t,\e)=O(\e^{-1/2})$ uniformly in $\hat t$.
The sum in \eqref{G and a_k b_k} can now be written as a double sum
\be\label{dbl sum}
\begin{split}
&W_{\e,m}(\tilde\theta,\hat t):=
\Delta\al
\sum_{n=1}^{N(\hat t,\e)} (\pm 1) \sum_{k \in \Xi_n(\hat t,\e)}
G_m(\e^{-1/2}\tilde\theta,\hat t,\al_k) e\left(-m q_k\right).
\end{split}
\ee
In particular, the factor $+1$ or $-1$ is the same for all the terms in the same class.

Pick any $n$ and consider the set $\Xi_n(\hat t,\e)$. As is easily seen, this is a set of consecutive $k$'s, so
$\Xi_n(\hat t,\e)=\{k_f,k_f+1,\dots,k_l\}$ for some $k_f\le k_l$. Pick any $s^*\in (s_k,s_k+\tilde\theta)\cap H_0^{-1}(\hat t)$, where $k\in\Xi_n(\hat t,\e)$ is arbitrary. Then $s^*-\tilde\theta < s_k <s^*$ for all $k\in\Xi_n(\hat t,\e)$, and by \eqref{theta deriv min}
\be\label{del-k}
k_l-k_f \le \e^{-1/2}\tilde\theta/(\kappa|\theta'(\al^*)|)+O(\tilde\theta^2),
\ee
where the term $O(\tilde\theta^2)$ is uniform in $n$ and $\hat t$, and $\al^*$ solves $s^*=\e^{-1/2}\theta(\al^*)$. If the right-hand side of \eqref{del-k} drops below $1$, then $k_f=k_l$. Now we estimate $W_{\e,m}$ using \eqref{theta deriv min}, \eqref{psi_psipr}, \eqref{G and a_k b_k}, and \eqref{dbl sum}:
\be\label{W est}
\begin{split}
&|W_{\e,m}(\tilde\theta,\hat t)|\le
c\frac{\e^{1/2}}{1+m^2}\frac{\max(1,\e^{-1/2}\tilde\theta)}{1+(\tilde\theta^2/\e)}.
\end{split}
\ee
Since the set of $\hat t$ values (the range of $H_0$) is uniformly bounded, the same estimate as in \eqref{W est} applies to $B_m(\tilde\theta,\e)$, and we find
\be\label{last bnd}
\sum_m \e^{-1/2}\int_{0}^{\de} |B_m(\tilde\theta,\e)|\dd \tilde\theta
=O(\e^{1/2}\ln(1/\e)).
\ee
The case $\tilde\theta<0$ can be considered in the same fashion, and the lemma is proven.

The estimate in \eqref{last bnd} is better than $O(\e^{\ga/2})$ if $\ga<1$ (cf \eqref{partI-three}). Using only the smoothness of $H_0$ (i.e., \eqref{holder}), which implies that $H_0(s_k+\tilde\theta)-H_0(s_k)=O(|\tilde\theta|^\ga)$, we easily recover the result \eqref{partI-three}. Hence a truly novel mechanism (e.g., based on the consideration of level sets of $H_0$) is needed to establish the faster decay rates in \eqref{extra terms}. 

\section{Example of a monotone, Holder continuous function, which is non-differentiable in a dense set}
\label{sec:bad fn}

Set
\be\label{bad fn}
H_0(s)=\sum_{n=0}^\infty \varphi(2^n s)/3^n,\ \varphi(s):=\lfloor s\rfloor+\{s\}^{\ga},
\ee
where $0<\ga<1$. Obviously, the series above converges absolutely for any fixed $s$, and $H_0$ is strictly monotonically increasing. 

Next we show that $H_0$ is locally Holder continuous with exponent $\ga$. We have 
\be\label{aux-ineq}
|\varphi(s_2)-\varphi(s_1)|\le c\max(|s_2-s_1|^\ga,|s_2-s_1|),\ s_1,s_2\ge0,
\ee
for some $c$. The case $|s_2-s_1|\ge 1$ is obvious, so we assume $|s_2-s_1|\le 1$ and show that $\varphi(s_2)-\varphi(s_1)=O(|s_2-s_1|^\ga)$. Clearly, suffices it to consider the case $s_1=n-1+r_1$, $s_2=n+r_2$, where $0\le r_{1,2}< 1$. We will show that
\be\label{aux-1}
(n+r_2^\ga)-(n-1+r_1^\ga)\le (1+r_2-r_1)^\ga.
\ee
The case $r_2\ge r_1$ is obvious (and should not be considered anyway, because $s_2-s_1\ge 1$ in this case), so we assume $r_2=r_1-h$, where $0\le h\le r_1$. Then \eqref{aux-1} becomes
\be
(r_1-h)^\ga+1-r_1^\ga\le (1-h)^\ga.
\ee
Differentiating the left-hand side we see that it is increasing as a function of $r_1\in [h,1]$. Setting $r_1=1$ shows that the inequality holds. 

Pick any $h>0$ and consider the difference
\be
H_0(s+h)-H_0(s)=\left(\sum_{n\ge0:2^n h <1}+\sum_{n:2^n h \ge 1}\right) \frac{\varphi(2^n (s+h))-\varphi(2^n s)}{3^n}=:S_1+S_2.
\ee
By \eqref{aux-ineq},
\be
|S_1|\le c h^\ga \sum_{n\ge0:2^n h <1}2^{\ga n}/3^n\le ch^\ga,\
|S_2|\le c h \sum_{n:2^n h \ge 1}2^{n}/3^n\le ch,
\ee
which proves Holder continuity. Here $c$ denote various constants, which can be different in different places.

Finally we show that $H_0(s)$ is not Holder continuous with any exponent $\ga'>\ga$ on any interval. Pick any $j,m\in\mathbb N$ and set $s=j2^{-m}$. Pick any $h\in(0,2^{-m})$. Using that $\varphi$ is increasing gives
\be\label{finite-diff}
\frac{H_0(s+h)-H_0(s)}{h^{\ga'}} \ge \frac{\varphi(2^m(s+h))-\varphi(2^m s)}{3^m h^{\ga'}}
=\frac{(j+(2^m h)^\ga)-j}{3^m h^{\ga'}}=ch^{\ga-\ga'},
\ee
and the desired assertion follows because diadic integers are dense in $\br$.

\bibliographystyle{plain}
\bibliography{My_Collection}
\end{document}